\documentclass[12pt,a4paper,psamsfonts]{amsart}
\usepackage{amssymb,amscd,amsxtra,calc}
\usepackage{cmmib57}

\theoremstyle{plain}
    \newtheorem{thm}{Theorem}[section] %
    \renewcommand{\thethm}%
    {\arabic{section}.\arabic{thm}}
    \newtheorem{cor}[thm]{Corollary}
    \newtheorem{prop}[thm]{Proposition}
    \newtheorem{lem}[thm]{Lemma}
    \newtheorem*{thmA}{Theorem A}
    \newtheorem*{thmB}{Theorem B}
    \newtheorem*{thmC}{Theorem C}
    \newtheorem*{thmD}{Theorem D}
    
\theoremstyle{definition}
    \newtheorem{dfn}[thm]{Definition}
    \newtheorem{nota}[thm]{Notation}
    \newtheorem*{nottn}{Notation}
\theoremstyle{remark}
    
    \newtheorem{rem}[thm]{Remark}
    \newtheorem*{remn}{Remark}
    \newtheorem{exam}[thm]{Example}
    \newtheorem{conj}[thm]{Conjecture}
    \newtheorem{fact}[thm]{Fact}

\newcommand{\BCC}{\mathbb{C}}
\newcommand{\BGG}{\mathbb{G}}

\newcommand{\BPP}{\mathbb{P}}
\newcommand{\BQQ}{\mathbb{Q}}
\newcommand{\BRR}{\mathbb{R}}
\newcommand{\BZZ}{\mathbb{Z}}

\newcommand{\SA}{\mathcal{A}}
\newcommand{\SC}{\mathcal{C}}
\newcommand{\SF}{\mathcal{F}}
\newcommand{\SH}{\mathcal{H}}
\newcommand{\SO}{\mathcal{O}}
\newcommand{\SR}{\mathcal{R}}
\newcommand{\SU}{\mathcal{U}}

\newcommand{\SW}{\mathcal{W}}

\newcommand{\ep}{\varepsilon}

\newcommand{\id}{\operatorname{id}}
\newcommand{\Hom}{\operatorname{Hom}}
\newcommand{\Aut}{\operatorname{Aut}}
\newcommand{\Bir}{\operatorname{Bir}}
\newcommand{\Bim}{\operatorname{Bim}}
\newcommand{\PGL}{\operatorname{PGL}}
\newcommand{\Ker}{\operatorname{Ker}}

\newcommand{\Alb}{\operatorname{Alb}} 
\newcommand{\OH}{\operatorname{H}} 
\newcommand{\OR}{\operatorname{R}} 
\newcommand{\Hilb}{\operatorname{Hilb}} 
\newcommand{\Chow}{\operatorname{Chow}} 
\newcommand{\LS}[1]{|#1|} 
\newcommand{\SSpec}{\operatorname{\mathit{Spec}}}
\newcommand{\SProj}{\operatorname{\mathit{Proj}}}
\newcommand{\Km}{\operatorname{Km}} 
\newcommand{\Sym}{\operatorname{Sym}} 
\newcommand{\class}{\operatorname{cl}} 
\newcommand{\isom}{\simeq} 
\newcommand{\ratmap}
{{\,\cdot\negmedspace\cdot\negmedspace\cdot\negmedspace\to\,}}
\newcommand{\injmap}{\hookrightarrow}

\newcommand{\topo}{\mathrm{top}}
\newcommand{\red}{\mathrm{red}}
\newcommand{\coh}{\mathrm{coh}}
\newcommand{\can}{\mathrm{can}}
\newcommand{\rat}{\mathrm{rat}}
\newcommand{\qis}{\mathrm{qis}}
\newcommand{\alg}{\mathrm{alg}}
\newcommand{\reg}{\mathrm{reg}}
\newcommand{\apr}{\mathrm{apr}} 

\newcommand{\norm}[1]{\|#1\|}
\newcommand{\lquot}{\backslash} 
\newcommand{\rdn}[1]{\llcorner{#1}\lrcorner}
\begin{document}

\title[\'Etale endomorphisms of projective manifolds]{%
Building blocks of \'etale endomorphisms of complex projective manifolds}
\author{Noboru Nakayama}
\address[Noboru Nakayama]{%
\textsc{Research Institute for Mathematical Sciences} \endgraf
\textsc{Kyoto University, Kyoto 606-8502 Japan}}
\email{nakayama@kurims.kyoto-u.ac.jp}
\author{De-Qi Zhang}
\address[De-Qi Zhang]{%
\textsc{Department of Mathematics} \endgraf
\textsc{National University of Singapore, 2 Science Drive 2,
Singapore 117543}}
\email{matzdq@nus.edu.sg}

\begin{abstract}
\'Etale endomorphisms of complex projective manifolds
are constructed from two building blocks
up to isomorphism
if the good minimal model conjecture is true.
They are the endomorphisms of abelian varieties and
the nearly \'etale rational endomorphisms of weak Calabi-Yau varieties.
\end{abstract}


\thanks{The first named author is partly supported by the Grant-in-Aid for
Scientific Research (C), Japan Society for the Promotion of Science.
The second named author is supported by an ARF of NUS}

\maketitle


\section{Introduction}
We work over the field $\BCC$ of complex numbers.
In this paper, we shall give a systematic study of
\'etale endomorphisms of nonsingular projective varieties.
The \'etaleness assumption is quite natural because every surjective
endomorphism of $X$ is \'etale provided that $X$ is a nonsingular
projective variety and is non-uniruled
(cf.\ Remark~\ref{rem:IitakaFujimoto} below).
In the study of birational classification of algebraic varieties, we
usually have the following three reductions, where \( \kappa \)
denotes the Kodaira dimension and \( q \) the irregularity:
\begin{enumerate}
\item[(A)]  Varieties of \( \kappa > 0 \) \( \Rightarrow \)
varieties of \(\kappa = 0\), by the Iitaka fibration.

\item[(B)]  Varieties of \(\kappa = 0\) \( \Rightarrow \) abelian varieties and
varieties with \(\kappa = q = 0\), by the Albanese map.

\item[(C)]  Uniruled varieties \(\Rightarrow\) non-uniruled varieties,
by the maximal rationally connected fibration
(cf.\ \cite{Cp} and \cite{KoMM}).
\end{enumerate}

We want to show that there are similar reductions in the study of
\'etale endomorphisms of nonsingular projective varieties.
Theorems A, B, and C below correspond to the reductions (A), (B), and (C),
respectively. See \cite{Zh} for the case of automorphisms.

\subsection{The reduction (A)}
\label{subsect:Intro A}

This reduction is based on the Iitaka fibration. Let \( X \) be a
nonsingular projective variety with \( \kappa(X) > 0 \) and
\( f \colon X \to X \) a surjective morphism.
From a standard argument of pluricanonical systems, we infer that
\( f \) induces a birational automorphism \( g \)
of the base space \( Y \) of the Iitaka fibration \( X \ratmap Y \).
Theorem A below shows that the order of \( g \) is finite.
This was conjectured in several papers
(cf.\ \cite[Proposition 6.4]{AC}, \cite[Proposition 3.7]{FN}).
Theorem A treats not only holomorphic surjective endomorphisms of projective
varieties of \( \kappa > 0 \) but also
dominant meromorphic endomorphisms of
compact complex manifolds of \( \kappa > 0 \)
in the class \( \SC \) in the sense of Fujiki \cite{Fujiki_C}.
Note that a compact complex manifold is in the class \( \SC \) if and
only if it is bimeromorphic to a compact K\"ahler manifold
(cf.\ \cite{Va}).

\begin{thmA}
Let $X$ be a compact complex manifold in the class \( \SC \)
of \( \kappa(X) \geq 1 \)
and let $f \colon X \ratmap X$ be a dominant meromorphic map.
Let \( W_{m} \) be the image of the \( m \)-th pluricanonical map
\[\Phi_{m} \colon X \ratmap \LS{mK_{X}}^{\vee}
= \BPP(\OH^{0}(X, mK_{X}))\]
for a positive integer \( m \) with \( \LS{mK_{X}} \ne \emptyset \).
Then there is an automorphism \(g\) of \( W_{m} \) of
finite order such that
\(\Phi_{m} \circ f = g \circ \Phi_{m}\).
\end{thmA}

\begin{remn}
\begin{enumerate}
\item  If \( f \) is holomorphic, then, resolving
the indeterminacy locus of $\Phi_{m}$, we may assume that both
$f \colon X \to X$ and
$\Phi_m \colon X \to W_m$ are holomorphic
so that $\Phi_m \circ f = g\circ \Phi_m$.
This is because $f$ is \'etale and we can take an equivariant
resolution (in the sense of Section~\ref{subsect:equiv resol} below)
of the graph of Iitaka fibration
(cf.\ the proof of Lemma~\ref{lem:1}).

\item
Theorem A is known to be true by Deligne and Nakamura-Ueno
when \(X\) is Moishezon and
$f$ is a bimeromorphic automorphism
(cf.\ \cite[Theorem 14.10]{Ue}, \cite{NU}).
\end{enumerate}
\end{remn}

\subsection{The reduction (B)}
\label{subsect:intro B}

For a compact K\"ahler manifold \(M\) with \(c_{1}(M)_{\BRR} = 0\), we have
a finite \'etale cover \(\widetilde{M} \to M\) such that
\(\widetilde{M} \isom T \times F\) for a complex torus \(T\) and
a simply connected manifold \(F\) with \(c_{1}(F) = 0\),
by Bogomolov's decomposition theorem (cf.\ \cite{Bo}, \cite{Be}).
For a normal projective variety \( V \) with only
canonical singularities and with torsion \(K_{V}\),
we have the following weak decomposition by Kawamata
\cite[Corollary 8.4]{Ka85}:
\emph{There exists a finite \'etale covering \( F \times A \to V \)
for a weak Calabi-Yau variety \( F \) and
an abelian variety \( A \).}
Here, a normal projective variety \(F\) is called
\emph{weak Calabi-Yau} if
\( F \) has only canonical singularities, \( K_{F} \sim_{\BQQ} 0 \), and
\[ q^{\max}(F) := \max\{ q(F') \mid F' \to F \text{ is finite \'etale}
\} = 0 \]
(cf.\ Section~\ref{subsect:AlbClosure}).
If \( F \) is a nonsingular weak Calabi-Yau variety,
then \( \pi_1(F) \) is finite by
Bogomolov's decomposition theorem, so a finite \'etale cover
of \(F \) is expressed as a product of holomorphic symplectic manifolds
and Calabi-Yau manifolds.

In order to study the surjective endomorphisms of
a nonsingular projective variety \(X\) with \( \kappa(X) = 0 \),
we assume the existence of a good minimal model \(V\) of \( X \);
but we allow the variety \( V \) to have canonical singularities
as in \cite{Ka85}.
Then it has a meaning to consider the reduction of the endomorphisms
to those of weak Calabi-Yau varieties \( F \) and those of
abelian varieties \(A\) by
an \'etale covering \( F \times A \to V \).
Unfortunately, a holomorphic surjective endomorphism of \( X \) induces only
a rational map \(V \ratmap V\), but it satisfies the condition of
\emph{nearly \'etale map}, which is introduced
in Section~\ref{sect:nearlyEtale}.
Therefore, Theorem B below is meaningful for the reduction of type (B):

\begin{thmB}
Let \( V \) be a normal projective variety with only canonical
singularities such that \( K_{V} \sim_{\BQQ} 0 \).
Let \( h \colon V \ratmap V \) be a dominant rational map which is
nearly \'etale.
Then there exist an abelian variety \( A \),
a weak Calabi-Yau variety \( F \),
a finite \'etale morphism \( \tau \colon F \times A \to V \),
a nearly \'etale dominant rational map
\( \varphi_{F} \colon F \ratmap F \), and
a finite \'etale morphism \( \varphi_{A} \colon A \to A \) such that
\( \tau \circ (\varphi_{F} \times \varphi_{A}) = h \circ \tau \),
i.e., the diagram below is commutative\emph{:}
\[ \begin{array}{ccc}
F \times A & \overset{\varphi_{F} \times \varphi_{A}}{\ratmap} &
F \times A \\
\mbox{\scriptsize $\tau$} \downarrow \phantom{\tau} & & \phantom{\tau}
\downarrow \mbox{\scriptsize $\tau$} \\
V &\overset{h}{\ratmap} & \phantom{.}V.
\end{array}  \]
\end{thmB}

\begin{remn}
\begin{enumerate}
\item
By the proof, the commutative diagram above is birationally Cartesian.

\item If the algebraic fundamental group \( \pi_{1}^{\alg}(F) \) is
finite, then \( \varphi_{F} \) is a birational automorphism
(cf.\ Section~\ref{subsect:conj_discuss}).
In particular, if \( V \) has only terminal singularities and \(
q^{\max}(V) = \dim V - 2 \),
then \( \varphi_{F} \) is an automorphism, since \(F\) is a K3 surface
or an Enriques surface.
\end{enumerate}
\end{remn}

\subsection{The reduction (C)}
\label{subsect:intro C}

Let \( X \) be a uniruled nonsingular projective variety.
A \emph{maximal rationally connected fibration} of \( X \) in the
sense of \cite{Cp} and \cite{KoMM} is obtained by a certain rational
map \(X \ratmap \Chow(X) \)
into the Chow variety \( \Chow(X)  \), which
assigns a general point \(x \in X\)
a maximal rationally connected subvariety containing \( x \).
Let \( Y  \) be the normalization of the image of
\( X \ratmap \Chow(X) \) and
let \( \pi \colon X \ratmap Y \) be the induced rational fibration.
Assume that \( X \) admits an \'etale endomorphism
\( f \colon X \to X \). Then
there is an endomorphism \( h \colon Y \to Y \) such that
\( \pi \circ f = h \circ \pi \) (cf. Lemma \ref {lem:1}).
Since rationally connected manifolds are simply connected, the
endomorphism \( f \) is induced from \( h \).
In Theorem C below, we shall show that \( h \) is nearly \'etale.

\begin{thmC}
Let $X$ be a projective manifold with an \'etale endomorphism $f$.
Assume that \( X \) is uniruled.
Then there exist a projective manifold \( M \) with
an \'etale endomorphism \( f_{M} \colon M \to M\),
a non-uniruled normal projective variety \( Y \) with
a nearly \'etale endomorphism \( h \colon Y \to Y \),
a birational morphism \( \mu \colon M \to X \), and
a surjective morphism \( \pi \colon M \to Y \) such that
\begin{enumerate}
\item \( \pi \circ f_{M} = h \circ \pi\), \( \mu \circ f_{M} = f \circ \mu \),
i.e., the diagram below is commutative\emph{:}
\[ \begin{CD}
Y @<{\pi}<< M @>{\mu}>> X \\
@V{h}VV @V{f_{M}}VV @V{f}VV \\
Y @<{\pi}<< M @>{\mu}>> \phantom{,}X,
\end{CD} \]
\item \( \deg f = \deg f_{M} = \deg h \),
\item \( \pi \circ \mu^{-1} \colon X \ratmap M \to Y \) is birational to
the maximal rationally connected fibration of \( X \).
\end{enumerate}
\end{thmC}

\begin{remn}\hfill
\begin{enumerate}
\item
$h^{-1}(Y_{\rat}) = Y_{\rat}$ and the
restriction $Y_{\rat} \to Y_{\rat}$ of \( h \) is \'etale for
the open subset $Y_{\rat} \subset Y$ consisting of
the smooth points and the points of rational singularity,
by Proposition~\ref{prop:nEtendo} and
Proposition~\ref{prop:nEt}, \eqref{prop:nEt:item2}.

\item If $Y$ has the relative canonical model
$Y_{\can}$ for resolutions of singularities of $Y$,
then, by Lemma~\ref{lem:fg->et},
$h$ lifts to an \'etale endomorphism of
$Y_{\can}$ and also to an \'etale endomorphism of a certain
resolution $Y'$ of singularities of $Y$.
The recent paper \cite{BCHM} has announced a proof of the existence of
minimal models of varieties of general type even in a relative
setting.
The existence of the relative canonical model $Y_{\can}$ follows
from the result.
\end{enumerate}
\end{remn}

\subsection{Equivariant resolutions}
\label{subsect:equiv resol}

Let \(V\) be a normal projective complex variety and
\( f \colon V
\linebreak
\to V \) an \'etale endomorphism.
Then there exists a resolution \( \mu \colon X \to V \)
of singularities such that the induced rational map
\( \mu^{-1} \circ f \circ \mu \colon X \to V \to V \ratmap X \)
is a holomorphic \'etale endomorphism of \( X \).
This is known as the existence theorem of an equivariant resolution
when \( f \) is an automorphism.
However, the proof is also effective for non-isomorphic \'etale
endomorphisms:
A method of resolution of singularities
is called to have a \emph{functoriality} if, for any smooth morphism
\(X \to Y\), and for the resolutions of singularities
\(X' \to X\) and \(Y' \to Y\) given by the method,
\(X'\) is isomorphic to the fiber product \(X \times_{Y} Y'\).
The recent methods by Bierstone-Milman and by Villamayor using
the notion of \emph{invariant} have the functoriality
(cf.\ \cite{BM}, \cite{EH}, \cite{EV}, \cite{JW}, \cite{KoR}).
Therefore, we call the resolution \( X \to V \) above also an
\emph{equivariant resolution} even if \( f \) is a non-isomorphic \'etale
endomorphism of \( V \).

\subsection{The meaning of our reduction}
\label{subsect:meaning}

Let \( X \) be a nonsingular projective variety with an \'etale
endomorphism \( f \).

First, assume that \( X \) is uniruled.
In view of Theorem~C, \( f \) is considered to be
built from a nearly \'etale endomorphism \( h \)
of a non-uniruled normal variety \( Y \)
up to isomorphism.
Further, we can replace \( Y \) with a nonsingular variety and
\( h \) with an \'etale endomorphism, provided that the minimal model
conjecture is true for varieties birational to \( Y \).

It is conjectured that a projective variety \( X \) is non-uniruled
if and only if \( \kappa(X) \geq 0 \).
This is regarded as a weak version of the abundance conjecture, and the
three-dimensional case is proved by Miyaoka \cite{MiAnn},
which is a key to the proof of
the three-dimensional abundance conjecture by \cite{MiAb} and \cite{Ka92}.
The good minimal model conjecture is the combination of the minimal
model conjecture and the abundance conjecture.
Thus, under the assumption of good minimal model conjecture,
the study of \'etale endomorphisms
is reduced to that of \'etale endomorphisms
of varieties of \( \kappa \geq 0 \).

\begin{rem}\label{rem:IitakaFujimoto}
A surjective endomorphism of a non-uniruled projective manifold \(X\) is \'etale.
This follows from \cite[Theorem 2]{Ii77} in the case of \(\kappa(X) \geq 0\)
(cf.\ \cite[Theorem 11.7]{Iitaka}).
Fujimoto \cite[Lemma 2.3]{Fujimoto} gives
another proof in the same case. The proof also works in the case
where \(K_X H^{n-1} \geq 0\) for any ample divisor \(H\) of \(X\);
this condition is satisfied if \(X\) is not uniruled by
Miyaoka-Mori's criterion \cite{MiMo}.
Thus, the \'etaleness for non-uniruled manifolds follows.
\end{rem}

Second, assume that \( \kappa(X) > 0 \).
Then we have the Iitaka fibration \( \Phi \colon X \ratmap Y \).
By Theorem~A, \( f \) induces a birational automorphism \( g \) of \( Y \)
of finite order.
Replacing \( X \) with a birational model,
we may assume that \( \Phi \) is holomorphic as remarked
after Theorem~A.
Iterating \( f \), we may assume \( f \) to be a morphism over \( Y \),
i.e., \(\Phi \circ f = \Phi\).
Then \( f \) induces an \'etale endomorphism of a general fiber \( F \)
of \( \Phi \).
This is a kind of reduction of \( f \) to an endomorphism of
varieties of \( \kappa = 0 \).
Theorem~A reduces the dynamical study of holomorphic endomorphisms of compact
K\"ahler manifolds of \(\kappa > 0\) to the case
of \(\kappa = 0\), by the results in Appendix~\ref{append:ent}.
In fact, we show in Appendix~\ref{append:ent} (cf.\ Theorem~D)
that \( d_{1}(f) = d_{1}(f|_{F}) \) and \( h_{\topo}(f) =
h_{\topo}(f|_{F}) \) for the first dynamical degree
\( d_{1} \) and the topological entropy \( h_{\topo} \).

However, even if we know the endomorphisms of fibers very well, it is
rather difficult to determine the structure of \( f \), as in the
papers \cite{Fujimoto} and \cite{FN}, which have determined the structure
of endomorphisms of \(3\)-dimensional projective manifolds of
\( \kappa \geq 0 \).

Third, assume that \( \kappa(X) = 0 \).
As in Section~\ref{subsect:intro B}, an \'etale endomorphism \( f \)
of \( X \) induces a nearly \'etale rational endomorphism of a weak
Calabi-Yau variety \( F \) and an endomorphism of an abelian variety \( A \),
provided that the good minimal model conjecture is true.
However, it is not clear that the nearly \'etale endomorphism induces
an \'etale endomorphism of a certain resolution of singularities of \( F \).
Further, for the converse direction,
it is not easy to recover the original endomorphism \( f \)
from the two endomorphisms of \( F \) and \( A \)
(cf.\ \cite{Fujimoto} for the \( 3 \)-dimensional case).

Therefore, we can conclude, under the assumption of good minimal model
conjecture, that the nearly \'etale endomorphisms of
weak Calabi-Yau varieties and the endomorphisms of abelian varieties
are the building blocks of all the \'etale endomorphisms of complex projective
manifolds.

For non-\'etale surjective endomorphisms (necessarily on
uniruled manifolds), on the one hand, we know
many examples of rationally connected varieties admitting
non-isomorphic surjective endomorphisms, such as projective spaces,
toric varieties, etc. (cf.\ \cite{Na02}). On the other hand, at the moment,
we do not have structure theorems for
all the endomorphisms on them.

\subsection*{Organization of the paper}
Section~\ref{sect:Iitaka} is devoted to proving Theorem~A and the
application to the pluricanonical representation of the bimeromorphic
automorphism group (cf.\ Corollary~\ref{cor:NU}).
In Section~\ref{sect:nearlyEtale}, we introduce a key notion of nearly
\'etale map and study its elementary properties.
Sections~\ref{sect:alb} and \ref{sect:mrc} are devoted to Theorem~B
and Theorem~C, respectively.
In Appendix~\ref{append:ent},
we shall prove the equalities on the first dynamical degrees and
the topological entropies mentioned above.

\subsection*{Notation and terminology}
We refer to \cite{KMM} for the definitions of
minimal models and of singularities including terminal,
canonical, log-terminal, and rational singularities.
Also, we refer to \cite{Ko96}, \cite{KM}, \cite{Mo}
for additional information on the birational geometry
and the minimal model theory.
Here we add an explanation of
the notion of \emph{fibration} (or \emph{fiber space}):
Let \(\pi \colon X \to Y\) be a proper morphism of
normal algebraic (resp.\ complex analytic) varieties.
If a general fiber of \(\pi\) is connected, then every fiber
is connected and \(\SO_{Y} \isom \pi_*\SO_{X}\).
In this case, \( \pi \) is called a fibration or a fiber space.
A proper rational map
(cf.\ Definition~\ref{dfn:propratmap} below)
(resp.\ a proper meromorphic map) between normal varieties is
called a fibration or a fiber space if it is the
the composite of a proper birational (resp.\ bimeromorphic) map
and a holomorphic fibration.

\subsection*{Acknowledgement}
The first named author expresses his gratitude to the Department of
Mathematics of the National University of Singapore for the hospitality
during his stay in October 2006. The work of this paper is based on
the discussion with the second named author during the stay.
The authors are grateful to Dr.~Hiraku Kawanoue for answering questions on
equivariant resolutions, and to Professor Takeshi Abe for pointing out
a gap in the proof of Proposition~\ref{prop:overP1} in an earlier version.
The authors would like to thank Professor Yoshio Fujimoto
for his encouragement and the referee for the suggestions.

\section{The case of positive Kodaira dimension}
\label{sect:Iitaka}

\subsection{Iitaka fibration}
\label{subsect:pre ThA}

In the situation of Theorem A, we may assume that \( X \) is a compact
K\"ahler manifold, without loss of generality.
We have a natural isomorphism
\( f^{*} \colon \OH^{0}(X, mK_{X}) \xrightarrow{\isom} \OH^{0}(X, mK_{X})\).
In fact, there is a bimeromorphic morphism \( \mu \colon Z \to X \)
from another compact K\"ahler manifold \( Z \)
such that \( \varphi := f \circ \mu \colon Z \to X\) is holomorphic;
then \( f^{*} \) is expressed as the composite
\[ \OH^{0}(X, mK_{X}) \xrightarrow{\varphi^{*}} \OH^{0}(Z, mK_{Z})
\xrightarrow[\isom]{(\mu^{*})^{-1}} \OH^{0}(X, mK_{X}), \]
which does not depend on the choice of \( \mu \colon Z \to X \).
The pullback \( f^{*} \) induces an automorphism \( g \) of
\( \LS{mK_{X}}^{\vee} \) preserving \( W_{m} \).
The problem is to show the finiteness of the order of \( g \in \Aut(W_{m}) \).
We begin with the following simple result.

\begin{lem}\label{lem:infinite->ruled}
If \emph{Theorem A} does not hold, then there is a
positive-dimensional connected commutative linear algebraic subgroup
\( G \subset \Aut(W_{m})\) such that \( g^{k} \in G \) for some \( k > 0 \).
In particular, \(W_{m} \) is ruled in this case.
\end{lem}

\begin{proof}
Let \( \widetilde{G} \subset \PGL = \Aut(\LS{mK_{X}}^{\vee}) \) be
the Zariski closure of the cyclic group \( I = \{ g^{j} \mid j \in \BZZ\} \).
Then \( \widetilde{G} \) is abelian, since it is contained in the
abelian algebraic subgroup \( Z(I) \cap Z(Z(I)) \),
where \( Z(S) \) denotes the algebraic subgroup
\( \{\gamma \in \PGL \mid \gamma s = s \gamma \text{ for any }
s \in S\} \)
for a subset \( S \subset \PGL \).
Let \( G \) be the identity connected component of \( \widetilde{G} \).
Then \( \dim G > 0 \) and \( g^{k} \in G \) for some \( k > 0 \).
Since the action of \( G \) preserves \( W_{m} \), \( G \) acts
non-trivially on \( W_{m} \). Thus \( W_{m} \)
is ruled by a result of Matsumura (cf.\ \cite[Theorem 14.1]{Ue}).
\end{proof}

\begin{rem}
There is another proof of Lemma~\ref{lem:infinite->ruled} by an
argument similar to \cite[Proposition 14.7]{Ue}: In fact, we can show
that \( f^{*} \) is expressed as a
diagonal matrix. Thus, \( G \) is contained in an algebraic torus.
We can prove more on \( f^{*} \) by the argument of
\cite[Proposition 14.4]{Ue}: If \( \lambda \) is an
eigenvalue of \( f^{*} \colon \OH^{0}(X, mK_{X})  \to \OH^{0}(X, mK_{X})\),
then \( \lambda \) is an algebraic integer with \( |\lambda|^{2/m} = \deg f \).
\end{rem}

The following is a key to the proof of Theorem~A.

\begin{prop}\label{prop:overP1}
Let \( \pi \colon X \to Y \) be a fiber space of a compact K\"ahler
manifold \( X \) over a nonsingular rational curve \( Y \isom \BPP^{1} \).
Let \( f \colon X \ratmap X \) be a dominant meromorphic map and
\( g \colon Y \xrightarrow{\isom} Y \) an automorphism such that
\( \pi \circ f = g \circ \pi \).
If \( \kappa(X) \geq 0 \), then \( g \) is of finite order.
\end{prop}

\begin{proof}
\emph{Step}~1.
\emph{We first prove in the case where
\( p_{g}(X) = \dim \OH^{0}(X, K_{X}) > 0 \)}.
Since \( \pi \) is smooth over a dense open subset \( U \) of \( Y \),
we have a variation of Hodge structure
\( H_{U} = \OR^{d}\pi_{*}\BZZ_{X}|_{U} \) for
\( d = \dim X/Y = \dim X - 1  \).

Let \( \mu \colon Z \to X \) be a bimeromorphic morphism from another
compact K\"ahler manifold \( Z \) such that
\( \varphi := f \circ \mu \colon Z \to X\) is holomorphic.
Then
\[ f^{*}_{\omega} \colon g^{*}(\pi_{*}\omega_{X/Y})
\xrightarrow{\varphi^{*}}
g^{*}(\pi_{*}\varphi_{*} \omega_{Z/Y})
\xrightarrow[\isom]{(\mu^{*})^{-1}} \pi_{*}\omega_{X/Y}\]
is injective.
Note that \( \pi_{*}\omega_{X/Y} \) is just the \( d \)-th
filter \( \SF^{d}({}^{u}\SH) \) of the upper canonical extension \( {}^{u}\SH \) of
\( H_{U} \otimes \SO_{U}\) in the sense of Koll\'ar~\cite{KollarHD2}.

We have also the pullback homomorphism
\[ f^{*}_{\coh} \colon g^{-1}(\OR^{d}\pi_{*}\BZZ_{X})
\xrightarrow{\varphi^{*}}
g^{-1}\OR^{d} (\pi \circ \varphi)_{*}\BZZ_{Z}
\xrightarrow{\mu_{*}} \OR^{d}\pi_{*}\BZZ_{X}, \]
where \( \mu_{*} \) is induced from the trace map
\( \OR \mu_{*}\BZZ_{Z}[2n] \to \BZZ_{X}[2n] \) for \( n = \dim X \).
Note that
\( f^{*}_{\coh} \) is compatible with \( f^{*}_{\omega} \), i.e., the diagram
\[ \begin{CD}
(g^{-1}\OR^{d}\pi_{*}\BZZ_{X})|_{U'} \otimes \SO_{U'}
@>{f^{*}_{\coh}}>>
(\OR^{d}\pi_{*}\BZZ_{X})|_{U'} \otimes \SO_{U'} \\
@AAA @AAA \\
(g^{*}\pi_{*}\omega_{X/Y})|_{U'}
@>{f^{*}_{\omega}}>> \pi_{*}\omega_{X/Y}|_{U'}
\end{CD} \]
is commutative over the open subset \( U' = U \cap g^{-1}U \).
Let \( J_{k} \subset \OR^{d} \pi_{*}\BZZ_{X}\) be the image of
\begin{multline*}
(f^{*}_{\coh})^{k} \colon
(g^{k})^{-1}(\OR^{d}\pi_{*}\BZZ_{X}) \xrightarrow{(g^{k-1})^{-1}f^{*}_{\coh}}
(g^{k-1})^{-1}(\OR^{d}\pi_{*}\BZZ_{X}) \to \cdots \\
\cdots \to g^{-1}(\OR^{d}\pi_{*}\BZZ_{X}) \xrightarrow{f^{*}_{\coh}}
\OR^{d}\pi_{*}\BZZ_{X}
\end{multline*}
for \( k > 0 \).
Then \( J_{k} \supset J_{k+1} \) and
\( f^{*}_{\coh}(g^{-1}J_{k}) = J_{k+1} \).
Moreover, for \( k \gg 0 \),
\( f^{*}_{\coh} \colon g^{-1}J_{k} \to J_{k + 1} \) is isomorphic
and \( J_{k} \otimes \BQQ = J_{k+1} \otimes \BQQ \),
since any stalk of \( \OR^{d}\pi_{*}\BZZ_{X} \) is a finitely
generated abelian group.
We set \( J := J_{k}\) for certain \( k \gg 0 \). Then, for a non-empty
Zariski open subset \( U'' \subset U \),
\( J|_{U''} \) defines a variation of Hodge substructure of
\( \OR^{d}\pi_{*}\BZZ_{X}|_{U''} \), and
\( J|_{U''} \otimes \SO_{U''} \) contains
\( \pi_{*}\omega_{X/Y}|_{U''} \) as the \( d \)-th Hodge filter.
Furthermore, the injection \( f^{*}_{\coh} \colon g^{-1}J \injmap J \) is compatible
with \( f^{*}_{\omega} \).
Let \( U_{\max} \) be the maximal open subset of \( Y \) such that
there is a variation of Hodge structure \( J_{\max} \) on \( U_{\max} \)
with \( J_{\max}|_{U''} \isom J|_{U''} \).
Then \( g^{-1}U_{\max} = U_{\max} \).
Thus, we may assume that \( Y \setminus U_{\max} \)
consists of at most two points; otherwise, \( g \) is of finite order.
Note that a K\"ahler form of \( X \) defines a real polarization of
the variation of Hodge structure
\( J_{\max} \).

Suppose that \( U_{\max} = Y \isom \BPP^{1} \) or  \( U_{\max} \isom \BCC \).
Then \( U_{\max} \) is simply connected
and we have a period map from
\( U_{\max} \) to the period domain associated with the polarized
variation of Hodge structure \( J_{\max} \)
(we need the polarization here).
Griffiths proves that this map is ``horizontal'' (cf.\ \cite[Sections 8 and 9]{Gri})
and satisfies a kind of Schwarz inequality (\cite[Theorem 10.1]{Gri}).
Thus, this is a constant map by a reason similar to the fact that the map from \( \BCC \) to
a Kobayashi hyperbolic manifold is constant; in this case, the monodromy
representation is also trivial, so the variation of Hodge structure \( J_{\max} \) is constant,
i.e., the
local system is trivial and every Hodge filtration \( F^{p}(J_{\max}) \) is
a trivial locally free sheaf.
Hence, the sheaf \( \pi_{*}\omega_{X/Y} \isom F^{d}(J_{\max}) \) is trivial.
This leads to a contradiction: \( H^0(X, K_X) \isom
H^0(Y, \pi_*(\omega_X)) = H^0(Y, \pi_*(\omega_{X/Y}) \otimes \omega_Y)
= H^0(\BPP^1, \SO(-2)) = 0. \)

Hence, we may assume \( U_{\max} = \BCC^{\star} \).
Then, the period map associated with \( J_{\max} \) is constant,
since the universal covering space of \( U_{\max} \) is \( \BCC \).
In particular, the image of the monodromy representation
\( \pi_{1}(U_{\max}, y) \to \Aut(J_{\max, y})\) is finite
for a reference point \(y \in Y\).
Let \( \tau \colon Y' \isom \BPP^{1} \to Y \) be a cyclic covering
\'etale over \( U_{\max} \) such that \( \tau^{-1}J_{\max} \) extends to
a trivial constant sheaf of \( Y' \).
We may assume that \( g \) lifts to an automorphism \( g' \)
of \( Y' \).
Let \( X' \to X \times_{Y} Y' \) be a resolution of singularities and
let \( \pi' \colon X' \to Y' \) be the induced morphism.
Then \( f \times g' \) induces a meromorphic endomorphism of \( X' \)
and \( p_{g}(X') > 0 \).
Since \( \tau^{-1}J_{\max} \) has trivial monodromy, a similar
variation of Hodge structure \( J'_{\max} \) is defined on \( Y' \).
Thus \( \pi'_{*}\omega_{X'/Y'} \) is a free \( \SO_{Y'} \)-module,
and we have the same contradiction as above.

\emph{Step}~2. \emph{General case}:
Let \( s \in \OH^{0}(X, mK_{X}) \) be an eigenvector of \( f^{*} \).
We shall consider a cyclic covering corresponding
to taking the \( m \)-th root of \( s \):
Let \( \SA = \bigoplus_{i = 0}^{m-1} \SO_{X}(-iK_{X}) \) be
the \( \SO_{X} \)-algebra determined by
\( s \colon \SO_{X}(-mK_{X}) \to \SO_{X} \).
Taking a resolution \( \widehat{X} \) of singularities of
\( \SSpec_{X} \SA \), let \( \tau \colon \widehat{X} \to X \)
be the composite \(\widehat{X} \to \SSpec_{X} \SA \to X \).
Then \( \tau^{*}s \in \OH^{0}(\widehat{X}, mK_{\widehat{X}}) \) is
expressed as \( \sigma^{m} \) for a section
\( \sigma \in \OH^{0}(\widehat{X}, K_{\widehat{X}}) \).
Let \( X' \) be a connected component of \( \widehat{X} \).
Then \( \kappa(X') = \kappa(X) \) and \( p_{g}(X') > 0 \).
Let \( \pi' \colon X' \to Y' \) be the fiber space and let
\( \lambda \colon Y' \to Y \) be the finite morphism obtained as the
Stein factorization of \( X' \to X \to Y \).

Since \( s \) is an eigenvector of \( f^{*} \), we have
a meromorphic map \( \hat{f} \): \( \widehat{X} \ratmap \widehat{X} \) with
\( \tau \circ \hat{f} = f \circ \tau\).
Replacing \( f \) with a suitable power \( f^{k} \), we may assume
that \( \hat{f} \) preserves \( X' \).
Let \( f' \colon X' \ratmap X' \) be the induced rational map.
Then there is an automorphism \( g' \) of \( Y' \) such that
\(\pi' \circ f' = g' \circ \pi'\) and
\(\lambda \circ g' = g \circ \lambda\).
If \( g' \) is of finite order, then so is \( g \).
If the genus of \( Y' \) is greater than one, then \( g' \) is of
finite order.
Even if the genus of \( Y' \) is one, \( g' \) is of finite order
since \( g' \) preserves the ample invertible sheaf
\( \lambda^{*}\SO(1) \). If the genus of \( Y' \) is zero,
then \( g' \) is of finite order by \emph{Step}~1.
Thus, we are done.
\end{proof}

\begin{remn}
If \( f \) is a holomorphic endomorphism of \( X \) and
if \( X \) is projective, then
Proposition~\ref{prop:overP1} follows from \cite{VZ}, since \( \pi \)
has at least three singular fibers preserved by \( g \).
\end{remn}

\subsection{Proof of Theorem A}
\label{subset:pf ThA}
We first compare \( W_{m} \) and \( W_{ml} \) for positive
integers \( m \) and \( l \) with \( \LS{mK_{X}} \ne \emptyset \).
There is a natural rational map
\( \Psi_{m, ml} \colon W_{ml} \ratmap W_{m} \) such that \(
\Phi_{m} = \Psi_{m, ml} \circ \Phi_{ml} \). In fact, \( \Psi_{m, ml}
\) is defined by the
commutative diagram
\[ \begin{array}{ccccc}
W_{ml} & \injmap & \LS{mlK_{X}}^{\vee} & = & \BPP(\OH^{0}(X, mlK_{X})) \\
\raisebox{1ex}{\mbox{\scriptsize $\Psi_{m, ml}$}}
\overset{\vdots}{\downarrow}\phantom{\mbox{\scriptsize $\Psi_{m, ml}$}}
& & & & \phantom{\mbox{\scriptsize $\mu$}}\overset{\vdots}{\downarrow}
\raisebox{1ex}{\mbox{\scriptsize $\mu$}}  \\
W_{m} & \injmap & \LS{mK_{X}}^{\vee} & \overset{\iota}\injmap &
\phantom{,}\BPP(\Sym^{l}\OH^{0}(X, mK_{X})),
\end{array} \]
where \( \iota \) is the Veronese embedding and \( \mu \) is induced
from the natural homomorphism
\( \Sym^{l}\OH^{0}(X, mK_{X}) \to \OH^{0}(X, mlK_{X}) \).
Let \( g_{m} \) and \( g_{ml} \), respectively,
be the automorphisms of
\( W_{m} \) and \( W_{ml} \) induced from \( f \).
Then \(g_{m} \circ \Psi_{m, ml} = \Psi_{m, ml} \circ g_{ml} \)
by construction.
Thus, we may replace \( m \) with any multiple \( ml \)
in order to prove Theorem~A.
Therefore, we may assume from the beginning
that \( \Phi_{m} \colon X \ratmap W_{m} \)
gives rise to the Iitaka fibration of \( X \).

We shall derive a contradiction by
assuming that \( g := g_{m} \) is of infinite order.
By Lemma~\ref{lem:infinite->ruled}, we may assume that
\( g \) is contained in a positive-dimensional
connected commutative linear algebraic group
\( G \subset \Aut(W_{m}) \).
Let \( Y \to W_{m} \) be an equivariant resolution of singularities
so that \( G \) acts on \( Y \) holomorphically.
There is a sequence
\[ \{\id\} = G_{0} \subset G_{1} \subset G_{2} \subset \cdots \subset
G_{l} = G \]
of algebraic subgroups such that \( \dim G_{i}/G_{i - 1} = 1 \)
for \( 1 \leq i \leq l \).
In particular, \( G_{i}/G_{i - 1} \isom \BGG_{\mathtt{m}} \) or
\( \BGG_{\mathtt{a}} \).
Let \( Y \ratmap Y_{1} \subset \Hilb(Y) \) be the meromorphic
quotient of \( Y \) by \( G_{1} \) (cf.\ \cite[Theorem 4.1]{Fujiki_auto}).
Then \( G/G_{1} \) acts on \( Y_{1} \).
Replacing \( Y \) and \( Y_{1} \) by their nonsingular models, respectively,
we may assume that \( Y \to Y_{1} \) is holomorphic and \( G \)-equivariant.
Let \( Y_{1} \ratmap Y_{2} \subset \Hilb(Y_{1}) \)
be the meromorphic quotient of \( Y_{1} \) by \( G_{2}/G_{1} \), and
replace \( Y \), \( Y_{1} \), and \( Y_{2} \) by their nonsingular
models, respectively, so that \( Y \to Y_{1} \) and
\( Y_{1} \to Y_{2} \) are \( G \)-equivariant morphisms.
Continuing similar constructions, we have a sequence
of \( G \)-equivariant morphisms
\[ Y = Y_{0} \to Y_{1} \to \cdots \to Y_{l} \]
such that \( Y \to Y_{i} \) is
birational to the meromorphic quotient by \( G_{i} \).
Then, for \( 0 \leq i < l \),
a general fiber of \( Y_{i} \to Y_{i+1} \) is a smooth rational curve which is
the closure of an orbit of \( G_{i+1}/G_{i} \).
There is a similar and stronger assertion in \cite[Theorem 4.6]{L}.

Replacing \(X\) by its blowup, we may assume that
the composite \(X \ratmap W_m \ratmap Y \) is holomorphic.
Let us consider the composition
\( \phi_{i} \colon X \to Y \to Y_{i} \) for \( 1 \leq i \leq l\).
For a very general point \( y_{i} \in Y_{i} \), let \( F_{i} \subset X\)
be the fiber \( \phi_{i}^{-1}(y_{i})  \) and let
\( C_{i-1} \subset Y_{i - 1} \) be
the fiber of \( Y_{i - 1} \to Y_{i} \) over \( y_{i} \).
Then \( \kappa(F_{i}) > 0 \) for \( i > 0 \)
by the easy addition formula:
\[ \kappa(X) = \dim Y \leq \kappa(F_{i}) + \dim Y_{i}. \]

Now \( g \) acts on \( Y_{l} \) trivially.
Thus \( f \colon X \ratmap X \) is a meromorphic map over
\( Y_{l}\), and a dominant meromorphic map \( F_{l} \ratmap F_{l} \)
is induced. By Proposition~\ref{prop:overP1}, the
action of \( g \) on \( C_{l - 1} \) is of finite order.
Thus \( g^{k} \) acts on \( Y_{l - 1} \) trivially for some
\( k > 0 \). Hence, \( \{g^{kj} \mid j \in \BZZ\} \subset G_{l-1} \).
This contradicts that \( G \) is the identity component of the
Zariski closure of \( \{g^{j}\mid j \in \BZZ\} \). This completes the proof
of Theorem~A.
\hfill \qedsymbol

\vspace{2ex}

Theorem~A has an application to the pluricanonical representations of
the bimeromorphic automorphism group \( \Bim(X)\) of
compact complex manifolds \( X \) in the class \( \SC \).
The following result is proved in \cite[\S 14]{Ue} (cf.\ \cite{NU})
when \( X \) is a Moishezon manifold.

\begin{cor}\label{cor:NU}
Let \( X \) be a compact complex manifold in the class \( \SC \) and let
\[ \rho_{m} \colon \Bim(X) \to \Aut(\OH^{0}(X, mK_{X}))
\]
be the \( m \)-th pluricanonical representation
\emph{(cf.\ \cite[\S 14]{Ue})} for a positive integer \( m \) with
\( \OH^{0}(X, mK_{X}) \ne 0 \).
Then the image of the induced group homomorphism below is
finite\emph{:}
\[ \rho'_{m} \colon \Bim(X) \to
\PGL(\OH^{0}(X, mK_{X})) = \Aut(\OH^{0}(X, mK_{X}))/\BCC^{\star}. \]
\end{cor}

\begin{proof}
Let \( W_{m} \) be the image of the \( m \)-th pluricanonical map
\( \Phi_m \colon X \ratmap \LS{mK_{X}}^{\vee}  \). Then \( W_{m} \) is not
contained in any hyperplane of \( \LS{mK_{X}}^{\vee} \).
Thus, for \( \gamma \in \Bim(X) \),
the order of \( \rho'_{m}(\gamma) \) equals the order of the
action of \( \gamma  \) on \( W_{m} \).
By Theorem~A, the order of \( \rho'_{m}(\gamma) \) is finite
for any \( \gamma \in \Bim(X) \).
So \( \rho_{m}(\gamma) \) is expressed as a diagonal
matrix and its eigenvalues are
\( (\alpha\theta_{1}, \alpha\theta_{2}, \ldots, \alpha\theta_{k}) \)
for \( k = \dim \OH^{0}(X, mK_{X}) \),
for a constant \( \alpha \in \BCC^{\star} \),
and for roots \( \theta_{i} \) of unity, where \( \theta_{1} = 1 \).
By \cite[Proposition 14.4]{Ue} and by an argument in
\cite[Theorem 14.10]{Ue}, \( \alpha\theta_{i} \) are algebraic integers,
\(|\alpha| = 1 \), and the degree
\( [\BQQ(\alpha\theta_{i}) : \BQQ] \) of field extension
\( \BQQ(\alpha\theta_{i})/\BQQ \) is bounded above by a suitable
constant \( N \) which depends neither on \( i \) nor on \( \gamma \).
Thus the degree \( [\BQQ(\theta_{i}) : \BQQ] \) is bounded above by
\( N^{2} \) for any \( i \).
Hence, the order of \( \rho'_{m}(\gamma) \) is uniformly bounded.
Therefore, the image of \( \rho'_{m} \) is a finite group by a theorem
of Burnside (cf.\ \cite[Theorem 14.9]{Ue}).
\end{proof}

\section{Nearly \'etale maps}
\label{sect:nearlyEtale}

We introduce the notion of nearly \'etale map
and study its basic properties.

\begin{dfn}[cf.\ {\cite{Iitaka}}]\label{dfn:propratmap}
Let \( h \colon V \ratmap W \) be a rational
(resp.\ meromorphic) map between
algebraic (resp.\ complex analytic) varieties.
The map \( h \) is called \emph{proper} if the projections
\( p_{1} \colon \Gamma_{h} \to V \) and
\( p_{2} \colon \Gamma_{h} \to W\)
are both proper for the graph \( \Gamma_{h} \subset V \times W \).
For algebraic varieties \( V \) and \( W \), \( V \) is called
\emph{proper birational} to \( W \) if there exists
a proper birational map \( V \ratmap W \).
\end{dfn}

\begin{remn}
The first projection \( p_{1} \colon \Gamma_{h} \to X \) is
proper for any meromorphic map \( h \). In particular,
a bimeromorphic map is always proper.
\end{remn}

We extend the notion of Stein factorization to
the case of proper rational (resp.\ proper meromorphic) maps.

\begin{dfn}\label{Stein}
Let \(h \colon V \ratmap W\) be a proper rational
(resp.\ proper meromorphic) map from a normal variety \(V\) to
an algebraic (resp.\ complex analytic) variety \(W\).
Then there exists a finite morphism \(\tau \colon W^{\sharp} \to W\)
uniquely up to isomorphism over \(W\) such that
\(W^{\sharp}\) is normal and \(h = \tau \circ \varphi \)
for a rational (resp.\ meromorphic) fiber space
\(\varphi \colon V \ratmap W^{\sharp}\).
The factorization \(V \ratmap W^{\sharp} \to W\) of \( h \)
is called the \emph{Stein factorization}.
\end{dfn}

\begin{remn}
If \(h\) is holomorphic, then this is the usual Stein factorization,
where \(\tau\) is determined by the isomorphism
\(h_{*}\SO_{V} \isom \tau_{*}\SO_{W^{\sharp}}\).
In case \(h\) is not holomorphic, let \(\mu \colon V' \to V\) be
a proper birational (resp.\ bimeromorphic) morphism
such that \(V'\) is normal and \(h' = h \circ \mu \colon V'\to W\)
is holomorphic; then \(\tau\) is just given
by the Stein factorization of \(h'\).
The uniqueness of \(\tau\) in this case
follows from Zariski's main theorem.
If \(V\) and \(W\) are algebraic, then \(W^{\sharp}\) is
just the normalization of \(W\) in the function field \(\BCC(V)\).
\end{remn}

\begin{dfn}\label{dfn:nearlyEtale}
Let \( h \colon V \ratmap W \) be a proper
rational (resp.\ meromorphic) map between
algebraic (resp.\ complex analytic) varieties.
The map \( h \) is called \emph{nearly \'etale} if
there exist proper birational (resp.\ bimeromorphic) maps
\( \mu \colon Y \ratmap W \), \( \nu \colon X \ratmap V \)
and a morphism \( f \colon X \to Y \) such that
\begin{enumerate}
\item  \(X\) and \(Y\) are algebraic (resp.\ complex analytic) varieties,

\item  \( f \) is a finite \'etale morphism, and

\item  \( \mu \circ f = h \circ \nu \), i.e., the diagram below
is commutative:
\[ \begin{array}{ccc}
X & \overset{f}{\longrightarrow} & Y \\
\raisebox{1ex}{\mbox{\scriptsize $\nu$}}
\overset{\vdots}{\downarrow} \phantom{\nu}
& & \phantom{\mu} \overset{\vdots}\downarrow
\raisebox{1ex}{\mbox{\scriptsize $\mu$}} \\
V & \overset{h}{\ratmap} & \phantom{.}W.
\end{array} \]
\end{enumerate}
\end{dfn}

\begin{remn}\hfill
\begin{itemize}
\item A nearly \'etale map is dominant and generically finite.
\item  A proper birational (resp.\ bimeromorphic) map is nearly \'etale.
\item  A finite \'etale morphism is nearly \'etale.
\item  Any nearly \'etale rational (resp.\ meromorphic) map
\( V \ratmap W \) is the composition of a proper birational
(resp.\ bimeromorphic) map
\( V \ratmap W^{\sharp} \) and
a nearly \'etale finite morphism \( W^{\sharp} \to W \).
In fact, it is enough to take the Stein factorization of the
rational (resp. meromorphic) map
from the normalization of \(V\) to \(W\).
\end{itemize}
\end{remn}

We now study basic properties of nearly \'etale maps.

\begin{lem}\label{lem:nearlyEt vs pi1}
Let \( h \colon V \ratmap W \) be a nearly \'etale rational
\emph{(}resp.\ meromorphic\emph{)} map between normal varieties.
Suppose that
\( \pi_{1}^{\alg}(W) \isom \pi_{1}^{\alg}(Y) \) for a resolution
\( Y \to W \) of singularities of \(W\),
where \(\pi_1^{\alg}\) stands for the algebraic fundamental group.
Then the finite morphism \(W^{\sharp} \to W\) is \'etale for
the Stein factorization \(V \ratmap W^{\sharp} \to W\) of \(h\).
\end{lem}

\begin{proof}
The fundamental group \(\pi_1\) is a proper
birational (resp.\ bimeromorphic) invariant for nonsingular varieties.
Therefore, there is a finite \'etale covering \(W' \to W\)
corresponding to the finite-index subgroup
\(\pi_1(X) \subset \pi_1(Y) \) via the isomorphism \(\pi_1(Y) \isom \pi_1(W)\)
for the \'etale covering \(X \to Y\) in Definition~\ref{dfn:nearlyEtale}.
Since \(X\) is proper birational (resp.\ bimeromorphic) to \(W'\), we have
\(W' \isom W^{\sharp}\) over \(W\).
\end{proof}

As a consequence of Lemma \ref{lem:nearlyEt vs pi1}, we can show that
the composition of two nearly \'etale maps are also
nearly \'etale.
Our next result shows that a nearly \'etale finite morphism is turned to be
an \'etale morphism by a suitable base change.

\begin{cor}\label{cor:nearlyEt vs pi1}
Let \( V \to W \) be a nearly \'etale finite morphism between normal
varieties. Let \( Z \to W \) be a morphism from a normal variety \(Z\)
such that the image contains a non-empty open subset of \(W\) and
\( \pi_{1}^{\alg}(Z) \isom \pi_{1}^{\alg}(M) \) for a resolution
\( M \to Z \) of singularities of \(Z\).
Let \(\tilde{V}_Z\) be the normalization of \(V \times_{W} Z\) and
\(\tilde{V}_Z \to Z^{\sharp} \to Z\) the Stein factorization of
the morphism induced from the second projection \(V \times_{W} Z \to Z\).
Then \( Z^{\sharp} \to Z \) is \'etale.
\end{cor}

\begin{proof}
Let \(\mu \colon Y \to W\) be a proper birational (resp.\ bimeromorphic)
morphism from a nonsingular variety \(Y\).
By Lemma~\ref{lem:nearlyEt vs pi1},
we have a finite \'etale covering \(X \to Y\)
such that \(V \to W\) is obtained as the Stein factorization
of \(X \to Y \to W\).
There is a morphism \(Y' \to Y \times_W Z\) from a nonsingular variety \(Y'\)
such that the composite \(Y' \to Y \times_W Z \to Z\)
is proper birational (resp.\ bimeromorphic),
since the image of \(Z \to W\) contains a non-empty open subset.
Let \(X' \to Y'\) be the finite \'etale morphism obtained as the pullback of
\(X \to Y\) by \(Y' \to Y \times_{W} Z \to Y\).
We infer that any connected component of \(Z^{\sharp}\) is obtained as the Stein
factorization of the morphism \(X' \to Y' \to Z\) restricted to
a connected component of \(X'\).
Therefore, \(Z^{\sharp} \to Z\) is \'etale by Lemma \ref{lem:nearlyEt vs pi1}.
\end{proof}

\begin{dfn}[{cf.\ \cite[(7.1.2)]{Ko93}, \cite[Section 2]{Na99}}]
Let \( (V, P) \) be a germ of normal complex analytic singularity
and \( \mu \colon Z \to V \) a resolution of singularity.
\begin{itemize}
\item  \( (V, P) \) is called an
\emph{algebraically \( \pi_{1} \)-rational singularity}
if the algebraic fundamental group of \( \mu^{-1}(P) \) is trivial.

\item  \( (V, P) \) is a
\emph{\( \pi_{1} \)-rational singularity}
if the fundamental group of \( \mu^{-1}(P) \) is trivial.
\end{itemize}
Let \(W\) be a normal variety.
\begin{itemize}
\item \( W_{\reg} \) denotes the nonsingular locus of \( W \).
\item  \( W_{\rat} \) denotes the set of points \( P \in W\) such that
\( (W, P) \) is nonsingular or is a rational singularity.
\item  \( W_{\apr} \) denotes the set of
points \( P \in W \) such that \( (W, P) \) is
an algebraically \( \pi_{1} \)-rational singularity.
\end{itemize}
\end{dfn}

\begin{rem}\label{rem:apr->pi1alg}
\hfill
\begin{enumerate}
\item If \( (V, \Delta) \) is log-terminal (klt) at the point \( P \)
for a \( \BQQ \)-divisor \( \Delta \) with \( \rdn{\Delta} = 0 \),
then \( (V, P) \) is an algebraically \( \pi_{1} \)-rational
singularity by \cite[Theorem 7.5]{Ko93}, and is in fact a
\( \pi_{1} \)-rational singularity by \cite{Ta}.
\item  \( W_{\reg} \) and \( W_{\rat} \) are Zariski open dense subset of a normal variety
\( W \), but \( W_{\apr} \) is not necessarily open.
\item If \( W = W_{\apr} \), then
\( \pi_{1}^{\alg}(W) \isom \pi_{1}^{\alg}(W) \) for any resolution
\( Y \to V \) of singularities of \(W\). In particular, this holds
if \( (W, \Delta) \) is log-terminal for a \( \BQQ \)-divisor \( \Delta \)
with \( \rdn{\Delta} = 0 \) (cf.\ \cite[Theorem 7.8]{Ko93}).
\end{enumerate}
\end{rem}

The result below says that a nearly \'etale endomorphism induces an
\'etale endomorphism of a certain nonsingular model, provided that the
minimal model conjecture is true.

\begin{lem}\label{lem:fg->et}
Let \( h \colon V \to V \) be a nearly \'etale finite morphism for a
normal variety \( V \). Assume the existence of the
relative canonical model \( V_{\can} \) for the resolutions of singularities of
\( V \)\emph{;}
equivalently, assume that the sheaf
\[ \SR_{V} := \bigoplus\nolimits_{m \geq 0} \mu_{*}\SO_{M}(mK_{M}) \]
of relative canonical ring is a finitely generated \( \SO_{V} \)-algebra
for a resolution
\( \mu \colon M \to V \) of singularities of \(V\), where
\( V_{\can} = \SProj_{V} \SR_{V} \). Then
\(h\) lifts to an \'etale endomorphism of \(V_{\can}\) and
to an \'etale endomorphism of
a certain resolution of singularities of \(V\).
\end{lem}

\begin{proof}
Since \( V_{\can} \) has only \( \pi_{1} \)-rational singularities,
applying Corollary~\ref{cor:nearlyEt vs pi1} to \(V_{\can} \to V\),
we have a finite \'etale morphism
\( h_{\can} \colon V^{\sharp}_{\can} \to V_{\can} \) and
a proper birational (resp.\ bimeromorphic) morphism \(V^{\sharp}_{\can} \to V\) such that
\[ \begin{CD}
V^{\sharp}_{\can} @>>> V \\
@V{h_{\can}}VV @V{h}VV \\
V_{\can} @>>> V
\end{CD} \]
is commutative.
Then \( V^{\sharp}_{\can} \) is also the relative canonical model
for the resolutions of singularities of \(V\),
since it has only canonical singularities and its
canonical divisor is also relatively ample over \( V \).
So \( V^{\sharp}_{\can} \isom V_{\can} \) and \( h_{\can} \) is
regarded as an \'etale endomorphism of \( V_{\can} \).
Applying the equivariant resolution of singularities of \(V_{\can}\)
with respect to the \'etale endomorphism \(h_{\can}\) in
the sense of Section~\ref{subsect:equiv resol},
we have a lift to a certain resolution of singularities of \(V\).
\end{proof}

\begin{lem}[{cf.\ \cite[Theorem 5.2]{Ko93}}]%
\label{lem:simpconn}
For a commutative diagram
\[ \begin{CD}
X @>{f}>> Y \\
@V{\psi}VV @V{\pi}VV \\
V @>{h}>> W
\end{CD} \]
of proper surjective morphisms of normal complex analytic varieties,
\(h\) is \'etale provided that
the following conditions are satisfied\emph{:}
\begin{enumerate}
\item  \( X\) and \( Y \) are nonsingular.
\item \( \psi \) and \( \pi \) have connected fibers.
\item The induced morphism \( X \to V \times_{W} Y \) is an isomorphism
over a dense open subset of \( V \).
\item  \( f \) is a finite \'etale morphism,
and \( h \) is a finite morphism.
\item \label{lem:simpconn:cond5} \( \OR^{i}\pi_{*}\SO_{Y} = 0 \) and
\( \OR^{i}\psi_{*}\SO_{X} = 0 \) for any \( i > 0 \).
\end{enumerate}
\end{lem}

\begin{proof}
We may assume that \( \deg f = \deg h > 1 \). Let \( P \in W \) be
an arbitrary point. It is enough to show that the cardinality
\(\sharp h^{-1}(P) \) of \( h^{-1}(P) \) equals \( \deg h \).
For the proof, we may replace \( Y \) with another nonsingular
variety \(Y'\) by taking a bimeromorphic morphism \( Y' \to Y \).
In fact, the pullback \(f' \colon X' \to Y' \) of \( f \colon X \to Y \)
by \( Y' \to Y \) induces a similar commutative diagram
satisfying the same conditions. Thus, we may assume that
the reduced
structure \( E = \pi^{-1}(P)_{\red} \) of the fiber
\( \pi^{-1}(P) =Y \times_{W} \{P\} \) is a normal crossing divisor,
after blowing up \( \pi^{-1}(P) \).
Then \( f^{-1}E\) is also a reduced normal crossing divisor of
\( X \). Here, \( f^{-1}E \) is the disjoint union of the reduced structures
\( \widetilde{E}_{Q} \) of the fibers \( \psi^{-1}(Q) \) for
points \( Q \in h^{-1}(P) \).

We shall show \( \chi(E, \SO_{E}) = 1 \). In the natural commutative
diagram
\[ \begin{CD}
(\OR^{i}\pi_{*}\BCC_{Y})_{P} @>{\isom}>> \OH^{i}(E, \BCC) \\
@VVV @VVV \\
(\OR^{i}\pi_{*}\SO_{Y})_{P} @>>> \phantom{,}\OH^{i}(E, \SO_{E}),
\end{CD}\]
the right vertical arrow is surjective by the theory of mixed Hodge
structures on normal crossing varieties. Therefore,
\( \OH^{i}(E,\SO_{E}) = 0 \) for \( i > 0 \) by
the condition \eqref{lem:simpconn:cond5}.
In particular, \( \chi(E, \SO_{E}) = 1\). By the same argument,
we also have \( \chi(\widetilde{E}_{Q}, \SO) = 1 \).
On the other hand, \( \chi(f^{-1}E, \SO) = (\deg f)\chi(E, \SO_{E}) = \deg f\),
since \( f \) is \'etale. Therefore,
\[ \sharp h^{-1}(P) = \sum \chi(\widetilde{E}_{Q}, \SO) =
\chi(f^{-1}E, \SO) = \deg f = \deg h. \qedhere\]
\end{proof}

\begin{prop}\label{prop:nEt}
Let \( h \colon V \to W \) be a nearly \'etale finite morphism
between normal varieties. Then\emph{:}
\begin{enumerate}
\item \label{prop:nEt:item1}
\( h^{-1}(W_{\apr}) \subset V_{\apr} \) and \( h \) is \'etale
along \( h^{-1}(W_{\apr}) \).
\item  \label{prop:nEt:item2}
\( h \) is \'etale along \( V_{\rat} \).
\item \label{prop:nEt:item3}
\( h^{-1}(W_{\apr} \cap W_{\rat}) = V_{\apr} \cap V_{\rat} \).
\item \label{prop:nEt:item4}
\( h^{-1}(W_{\reg}) = V_{\reg} \).
\end{enumerate}
\end{prop}

\begin{proof}
\eqref{prop:nEt:item1} follows from Lemma~\ref{lem:nearlyEt vs pi1}
applied to an open neighborhood of a point in \(W_{\apr}\). Moreover,
\eqref{prop:nEt:item3} and \eqref{prop:nEt:item4} are derived from
\eqref{prop:nEt:item1} and \eqref{prop:nEt:item2}, since
\( h(V_{\rat}) \subset W_{\rat} \). Therefore, it suffices to show
\eqref{prop:nEt:item2}.

Let \(P\) be a point of \(V_{\rat}\). In order to prove the \'etaleness of \(h\)
at \(P\), we may replace \(W\) with an open neighborhood \(\SW\) of \(h(P)\)
and \(V\) with a connected component of \(h^{-1}(\SW)\) containing \(P\).
Thus, we may assume that \(W\) is simply connected and
\(h^{-1}(h(P)) = \{P\}\), set-theoretically.
Let  \( \mu \colon Y \to W \) be a resolution of singularities.
Then, by Lemma~\ref{lem:nearlyEt vs pi1}, we have a finite
\'etale covering \( f \colon X \to Y \)
and a bimeromorphic morphism \( \nu \colon X \to V \) such that
\( \mu \circ f = h \circ \nu  \).
Applying Lemma~\ref{lem:simpconn} to the commutative diagram
expressing \(\mu \circ f = h \circ \nu \),
we infer that \(h\) is an isomorphism. Thus, we are done.
\end{proof}

For a nearly \'etale endomorphism of a normal algebraic variety,
we have the following additional property:

\begin{prop}\label{prop:nEtendo}
Let \( h \colon V \to V \) be a nearly \'etale finite surjective
endomorphism of a normal algebraic variety \( V \).
Then \( h^{-1}(V_{\rat}) = V_{\rat} \).
\end{prop}

\begin{proof}
Let \( B \) be the complement of \( V_{\rat} \) in \( V \).
Then we have the inclusion \( h^{-1}(B) \subset B \) by Proposition~\ref{prop:nEt}.
The other inclusion follows from Lemma~\ref{lem:endoclosed} below.
\end{proof}

\begin{lem}\label{lem:endoclosed}
Let \( h \colon V \to V\) be a finite surjective endomorphism of
an algebraic variety \( V \).
Suppose \(h^{-1}(B) \subset B\) for a closed subset \(B \subset V\). Then
\(h^{-1}(B) = B\). Further, \( \Gamma \mapsto h(\Gamma) \) induces an automorphism
of the set \( I_{B} \)  of irreducible components \( \Gamma \) of \( B \).
The inverse map is given by \( \Gamma \mapsto h^{-1}(\Gamma) \).
\end{lem}

\begin{proof}
Let \( I_{B}^{(k)} \) be the set of
irreducible components of \( B \) of dimension \( k \). Then
\( I_{B} = \bigcup I_{B}^{(k)} \).
Let \( J_{B}^{(k)} \) be the subset of \( I_{B}^{(k)} \) consisting of
\( \Gamma \in I_{B} \) with \( h(\Gamma) \subset B \).
It is enough to show the following two properties for any \( k \):
\begin{itemize}
\item[$(1)_{k}$:] \( I_{B}^{(k)} = J_{B}^{(k)} \).
\item[$(2)_{k}$:] \( h_{*} \colon J_{B}^{(k)} \to I_{B}^{(k)} \)
given by \( \Gamma \mapsto h(\Gamma) \) is bijective.
\end{itemize}

We shall prove by descending induction on \( k \).
We set \( d = \dim B\). For \( \Gamma \in I_{B}^{(d)} \),
some \( \Gamma' \in I_{B}^{(d)} \) is contained in \( h^{-1}(\Gamma) \).
Here, \( \Gamma' \in J_{B}^{(d)} \), since \( h(\Gamma') = \Gamma \).
Hence \( h_{*} \colon J_{B}^{(d)} \to I_{B}^{(d)} \) is surjective.
Since \( I_{B}^{(d)} \) is a finite set,
\( J_{B}^{(d)} = I_{B}^{(d)} \) and \( h_{*} \colon J_{B}^{(d)} \to
I_{B}^{(d)} \) is bijective.

Next, assume that \( (1)_{l} \) and \( (2)_{l} \) hold
for any integer \( l > k \).
If \( \Gamma \in I_{B}^{(k)} \), then an irreducible component
\( \Delta \) of \( h^{-1}(\Gamma) \) dominates \( \Gamma \),
and \( \Delta \subset \Gamma' \) for some \( \Gamma' \in I_{B} \).
If \( \dim \Gamma' > k \), then \( h(\Gamma') \subset B \) by induction,
and hence \( \Gamma \subset h(\Gamma') \); this is a contradiction.
Thus \( \dim \Gamma' = k \), \( \Gamma' = \Delta \),
and \( \Gamma' \in J_{B}^{(k)} \).
Hence, \( h_{*} \colon J_{B}^{(k)} \to I_{B}^{(k)} \) is surjective.
Therefore, \( I_{B}^{(k)} = J_{B}^{(k)} \) and
\( h_{*} \colon J_{B}^{(k)} \to I_{B}^{(k)} \) is bijective.
Thus, we are done.
\end{proof}

If \( f \colon X \to X \) is a surjective endomorphism of a nonsingular
projective variety \( X \) with \( \kappa(X) \geq 0 \), then \( f \)
is \'etale. But if we drop the condition of nonsingularity, then we
can expect neither the \'etaleness nor even the nearly \'etaleness.
Indeed, we have:

\begin{exam}
Let \( A \) be an abelian surface and \( V \) the quotient space of
\( A \) by the involution \( \iota \colon A \ni a \mapsto -a \in A \) with
respect to a group structure of \( A \). The minimal resolution
of singularities of \( V \) is a K3 surface called the Kummer surface associated
with \( A \) and is denoted by \(\Km(A) \).
The variety \( V \) has only canonical singularities and \( K_{V} \sim 0 \).
Let \( \mu = \mu_{m} \colon A \to A \) be the multiplication map
\( a \mapsto ma \) by an integer \(m\) with \(|m| > 1\).
Then it descends to a non-isomorphic surjective
endomorphism \( \mu_{V} \colon V \to V \).
Here, \( \mu_{V} \) is not nearly \'etale since
\( \Km(A)  \) is simply connected.
\end{exam}

\section{The case of zero Kodaira dimension}
\label{sect:alb}

\subsection{Albanese closure}
\label{subsect:AlbClosure}

\begin{dfn}
For a normal projective variety \( V \),
we define \( q^{\max}(V) \) to be the supremum of the irregularities
\( q(V') = \dim \OH^{1}(V', \SO_{V'})\) for all the finite \'etale
coverings \( V' \to V \).
\end{dfn}

If \( X \) is nonsingular projective with \( \kappa(X) = 0 \),
then \( q(X) \leq  q^{\max}(X) \leq \dim X \) by \cite{Ka81}.
Let \( V \) be a normal projective variety with only canonical
singularities such that \( \kappa(V) = 0\). Let \( X \) be
a nonsingular projective variety birational to \( V \).
Then \( q(X) = q(V) \) and \( \Alb(X) \isom \Alb(V) \),
since \( V \) has only rational singularities. Furthermore,
\(\pi_{1}^{\alg}(X) \isom \pi_{1}^{\alg}(V) \) by
Remark~\ref{rem:apr->pi1alg}.
In particular, the category of finite \'etale coverings over \( X \) is
equivalent to that of finite \'etale coverings over \( V \).
Therefore, \( q^{\max}(X) = q^{\max}(V) \).

\begin{dfn}
Let \( F \) be a normal projective variety with only canonical
singularities such that \( K_{F} \sim_{\BQQ} 0\).
If \( q^{\max}(F) = 0 \), then \( F\) is called
a \emph{weak Calabi-Yau variety}.
\end{dfn}

Let \( V \) be a normal projective variety with only canonical
singularities such that \( K_{V} \sim_{\BQQ} 0\).  Then, by
\cite[Corollary 8.4]{Ka85},
there is a finite \'etale covering \( F \times A \to V \) with \( F \)
a weak Calabi-Yau variety and \( A \) an abelian variety.
Here, \( q^{\max}(V) = \dim A \).

The result below guarantees the uniqueness (up to isomorphism) of
minimal \'etale cover $V\sptilde$ of $V$ realizing
$q^{\max}(V)$ as $q(V\sptilde)$. A similar result is found in \cite{Be}.

\begin{prop}\label{alb-c}
Let \( V \) be a normal projective variety with only canonical
singularities such that \( \kappa(V) = 0\). Then there exists a
finite \'etale Galois covering \( V\sptilde \to V \), unique up to
\emph{(non-canonical)} isomorphisms, such that\emph{:}
\begin{enumerate}
\item  \( q(V\sptilde) = q^{\max}(V) \), and
\item  if \( V' \to V \) is a finite \'etale covering from a variety
\( V' \) with \( q(V') = q^{\max}(V) \), then \( V' \to V \) factors
through \( V\sptilde \to V \).
\end{enumerate}
\end{prop}

\begin{proof}
There is a finite \'etale Galois covering \( V_{0} \to V \) with
\(q^{\max}(V) = q(V_{0}) \). For the Galois group \( G_{0} \) of
\(V_{0} \to V \), let \( H_{0} \) be the kernel of the natural
homomorphism
\[ G_{0} \to \Aut(\OH_{1}(\Alb(V_{0}), \BZZ)). \]
We set \( V\sptilde \) to be the quotient space
\( H_{0} \lquot V_{0} \). Then \( H_{0} \lquot \Alb(V_{0})\) is an abelian variety and
\(V\sptilde \to H_{0} \lquot \Alb(V_{0}) \) is the Albanese map of
\(V^{\sptilde} \). In particular, \( q^{\max}(V) = q(V\sptilde) \).
Let \( V_{1} \to V \) be a finite \'etale Galois covering such that
\( V_{1} \to V \) factors as \( V_{1} \to V_{0} \to V \). For the
Galois group \( G_{1} \) of \( V_{1} \to V \), let
\( H_{1} \subset G_{1} \) be the kernel of
\[ G_{1} \to \Aut(\OH_{1}(\Alb(V_{1}), \BZZ)). \]
Then \( H_{1} \) is just the pullback of \( H_{0} \subset G_{0} \)
by \( G_{1} \to G_{0} \), since \( \Alb(V_{1}) \to \Alb(V_{0}) \) is
an isogeny. In particular, \( H_{1} \lquot V_{1} \isom H_{0} \lquot
V_{0} = V\sptilde\). Therefore, \( V\sptilde \) is independent of
the choice of Galois covering \( V_{0} \to V \).
For an arbitrary finite \'etale cover \( V' \to V \) with
\( q(V') = q^{\max}(V) \), we have a finite Galois cover
\( V_{0}\to V \) which factors through \( V' \to V \). Then the Galois
group \( G'_{0} \) of \( V_{0} \to V' \) acts on \( \Alb(V_{0})\)
and, acts on \( \OH_{1}(\Alb(V_{0}), \BZZ) \)
trivially, since \( \Alb(V_{0}) \to \Alb(V') \)
is an isogeny. Thus
\(G'_{0} \subset H_{0} \) and we have a factorization
\( V' \to V\sptilde \to V \).
\end{proof}

We call the Galois cover \( V\sptilde \to V\)
the \emph{Albanese closure}. The Albanese closures of birationally
equivalent varieties satisfying the conditions of Proposition \ref{alb-c}
are also birationally equivalent.

\subsection{Splitting endomorphisms}
\label{subsect:split_endo}

We shall show that any nearly \'etale rational endomorphism \( \varphi \)
of the product \( F \times A \) of a weak Calabi-Yau variety \( F \)
and an abelian variety \( A \) is split as the product \( \varphi_{F}
\times \varphi_{A} \) of a nearly \'etale rational endomorphism
\( \varphi_{F}\) of \( F \) and an \'etale endomorphism
\( \varphi_{A} \) of \( A \).
A slightly more general assertion is proved in
Proposition~\ref{prop:bir diagonal} below.
To begin with, we recall the following well-known result
(cf.\ \cite[Theorem (4.6)]{Hm87}):

\begin{lem}\label{lem:aut=0}
Let \( F \) be a normal projective variety such that \( q(F) = 0 \).
If \( F \) is not ruled, then \( \Aut(F) \) is discrete.
\end{lem}

\begin{proof}
Let \( \SH \) be a very ample invertible sheaf
of \( F \). For an automorphism \( f \) of \( F \) belonging to the
identity component \( \Aut^{0}(F) \), the invertible sheaf \(
f^{*}\SH \) is isomorphic to \( \SH \), since the tangent space of
the Picard scheme of \( F \) at the origin \( [\SO_{F}] \) is isomorphic to the
zero-dimensional vector space \( \OH^{1}(F, \SO_{F}) \). Let \( \Phi
\colon F \injmap \BPP^{N} \) be the embedding defined by the very
ample linear system \( |\SH| \). Then \( f \) induces an
automorphism \( \rho(f) \colon \BPP^{N} \to \BPP^{N} \) such that
\(\Phi \circ f = \rho(f) \circ \Phi\). The automorphism \( \rho(f) \)
is contained in a linear subgroup of \( \PGL(N+1, \BCC) \)
preserving \( \Phi(F) \).
Since \( F \) is not ruled, we infer that the linear subgroup
acts on \( F \) trivially. Therefore, \( f = \id_{F} \).
Hence, \( \Aut^{0}(F) = \{\id_{F}\} \).
\end{proof}

The following is the first splitting criterion for an \'etale morphism.

\begin{lem}\label{lem:diagonal}
Let \( F \) and \( F' \) be non-ruled normal projective varieties
such that \( q(F) = q(F') = 0 \) and \( \dim F = \dim F' \). Let
\(A \) and \( A' \) be abelian varieties with \( \dim A = \dim A'\).
Let \( \varphi \colon F \times A \to F' \times A' \) be a surjective
\'etale morphism. Then \( \varphi = \varphi_{1} \times \varphi_{2}
\) for surjective \'etale morphisms \( \varphi_{1} \colon F \to F'
\) and \( \varphi_{2} \colon A \to A' \).
\end{lem}

\begin{proof}
The second projections \( p_{2} \colon F \times A \to A \) and
\(p'_{2} \colon F' \times A' \to A' \) are the Albanese maps of
\( F \times A \) and \( F' \times A' \), respectively. Thus an
\'etale map \( \varphi_{2} \colon A \to A' \) is induced so that
\(\varphi_2 \circ p_2 = p'_2 \circ \varphi\).
So for any \( a \in A \), there is an
\'etale morphism \( \rho_{a} \colon F \to F' \) such that
\[ \varphi(x, a) = (\rho_{a}(x), \varphi_{2}(a)). \]
The collection \( \{\rho_{a}\} \) gives rise to a morphism from \( A
\) into the scheme \( \Hom(F, F') \) of morphisms from \( F \) to \(
F' \). For a surjective \'etale morphism \( \psi \colon F \to F' \),
the tangent space of \( \Hom(F, F') \) at the point \( [\psi] \in
\Hom(F, F')\) corresponding to \( \psi \) is isomorphic to
\[ \Hom_{\SO_{F}}(\psi^{*}\Omega^{1}_{F'}, \SO_{F})
\isom \Hom_{\SO_{F}}(\Omega^{1}_{F}, \SO_{F}) \isom \OH^{0}(F,
\Theta_{F})\] for the tangent sheaf \( \Theta_{F} \). In particular,
the dimension of the tangent space equals that of
\( \Aut^{0}(F) \).
So the tangent space is
zero by Lemma~\ref{lem:aut=0}. Hence, \( \rho_{a} \) is independent of
the choice of \( a \in A \). Thus, \( \varphi = \varphi_{1}
\times \varphi_{2} \) for \( \varphi_{1} = \rho_{a} \).
\end{proof}

The following is a partial generalization of
Lemma~\ref{lem:diagonal} in the birational case.

\begin{lem}\label{lem:bir diag1}
Let \( F \) and \( F' \) be non-ruled nonsingular
projective varieties
such that \( q(F) = q(F') = 0 \) and \( \dim F = \dim F' \). Let
\(A \) and \( A' \) be abelian varieties with \( \dim A = \dim A'\).
Let \( \varphi \colon F \times A \ratmap F' \times A' \) be a birational
map. Then \( \varphi = \varphi_{1} \times \varphi_{2} \) for
a birational map \( \varphi_{1} \colon F \ratmap F' \) and
an isomorphism \( \varphi_{2} \colon A \xrightarrow{\isom} A' \).
\end{lem}

\begin{proof}
There is an isomorphism \( \varphi_{2} \colon A \to A' \) such that
\( p'_{2} \circ \varphi = \varphi_{2} \circ p_{2}\) for the second
projections \( p_{2} \colon F \times A \to A \) and \( p_{2}' \colon
F' \times A' \to A' \). Hence, we may assume that \( A = A' \) and
\( \varphi_{2} = \id_{A} \). Then \( \varphi \) is a birational map
\( F \times A \ratmap F' \times A \) over \( A \). For a general
point \( a \in A \), we have a birational map \( f_{a} \colon F
\ratmap F' \) as the restriction of \( \varphi \) to \( F \times
\{a\} \). Thus, we may also replace \( F' \) with \( F \) by
a suitable \( f_{a} \).
Therefore, we may assume from the beginning
that \( \varphi \) is a birational map \( F \times A \ratmap F
\times A \) over \( A \).
Then \( \varphi \) induces a rational map
from \( A \) into the scheme \( \Bir(F) \) of
birational automorphisms studied in \cite{Hm87}.
By \cite[Theorem (2.1)]{Hm88}, we have
\( \dim \Bir(F) = 0 \), and hence
the map \( A \ratmap \Bir(F) \) is constant.
Therefore, \( \varphi = \varphi_{F} \times \id_{A} \) for
a birational map \( \varphi_{F} \colon F \ratmap F\).
\end{proof}

The next is a sufficient condition to split the variety into a
product.

\begin{lem}\label{lem:bir split}
Let \( V \) be a normal projective variety with only canonical
singularities such that \( K_{V} \sim_{\BQQ} 0 \). Suppose that
\( V \) is birational to \( F \times A \) for an abelian variety
\( A \) and a normal variety \( F \) with only canonical singularities
such that \( K_{F} \sim_{\BQQ} 0\) and \( q(F) = 0 \). Then
\( V \isom F'\times A \) for a variety \( F' \) birational to \( F \).
\end{lem}

\begin{proof}
The composition \( V \ratmap F \times A \to A \) with the second
projection is the Albanese map of \( V \). There is a finite Galois
\'etale covering  \( A' \to A \) from an abelian variety \( A' \)
such that \( V \times_{A} A' \isom F' \times A' \) over \( A' \) for
a variety \( F' \), by \cite[Theorem 8.3]{Ka85}. Then \( F' \) is
normal with only canonical singularities, \( K_{F'} \sim_{\BQQ} 0
\), \( q(F') = 0 \), and we have a birational map
\[ \varphi \colon
F \times A' = (F \times A) \times_{A} A' \ratmap V \times_{A} A'
\isom F' \times A' \]
over \( A' \).
By Lemma~\ref{lem:bir diag1}, \( \varphi = \varphi_{F} \times \id_{A'} \)
for a birational map \( \varphi_{F} \colon F \ratmap F' \),
since the irregularities of nonsingular models of \( F \) and \( F' \)
are both zero.
The action of the Galois group \( G \) of \( A' \to A \) on
\( F' \times A' \isom V \times_{A} A'\) is written as a diagonal action
by Lemma~\ref{lem:diagonal}. Moreover, it is compatible with the
action of \( G \) on \( F \times A' \) by \( \varphi \), where \( G \)
acts trivially on the first factor \( F \). Therefore, \( G \) acts
trivially on the first factor \( F' \), and hence, \( V \isom F' \times A \).
\end{proof}

The following is also a partial generalization of
Lemma~\ref{lem:diagonal}.

\begin{prop}\label{prop:bir diagonal}
Let \( F \) and \( F' \) be normal projective varieties with only
canonical singularities such that \( K_{F} \sim_{\BQQ} 0 \),
\(K_{F'} \sim_{\BQQ} 0 \), and \( q(F) = q(F') = 0 \). Let \( A \)
and \( A' \) be abelian varieties with \( \dim A = \dim A' \). Let
\(\varphi \colon F \times A \ratmap F' \times A' \) be a nearly
\'etale rational map such that \( p_{2}' \circ \varphi =
\varphi_{A} \circ p_{2} \) for an \'etale morphism \( \varphi_{A}
\colon A \to A' \) and for the second projections \( p_{2} \colon F
\times A \to A \) and \( p_{2}' \colon F' \times A' \to A' \). Then
\( \varphi = \varphi_{F} \times \varphi_{A} \) for a nearly \'etale
rational map \( \varphi_{F} \colon F \ratmap F' \).
\end{prop}

\begin{proof}
By Lemma \ref{lem:nearlyEt vs pi1}, \(\varphi\) is the composite
of a birational map \(F \times A \ratmap V^{\sharp}\) and
a finite \'etale covering \(V^{\sharp} \to F' \times A'\),
since \(F' \times A'\) has only canonical singularities
(cf.\ Remark~\ref{rem:apr->pi1alg}).
Then \(V^{\sharp} \simeq F^{\sharp} \times A\) for a normal projective
variety \(F^{\sharp}\) birational to \(F\) by Lemma \ref{lem:bir split}.
Further, the \'etale covering \( V^{\sharp} \to F' \times
A' \) is isomorphic to \( \psi \times \varphi_{A} \) for finite
\'etale morphisms \( \psi \colon F^{\sharp} \to F' \) and
\(\varphi_{A} \colon A \to A' \) by Lemma~\ref{lem:diagonal}. So
we have only to show that the birational map \( F \times A \ratmap
V^{\sharp} \isom F^{\sharp} \times A \) is the product
of a birational map \( F \ratmap F^{\sharp} \) and the identity map
\( A \to A \). This is done by Lemma~\ref{lem:bir diag1}, since
the irregularities of nonsingular models of \( F \) and \( F^{\sharp} \)
are both zero.
\end{proof}

\subsection{Proof of Theorem B}
\label{subsect:pf ThB}

As in Proposition \ref{prop:bir diagonal},
\(h\) is the composite of a birational map \( V \ratmap V^{\sharp} \)
and an \'etale covering \( \lambda \colon V^{\sharp} \to V \).
Let \( \delta \colon V\sptilde \to
V \) be the Albanese closure of \( V \) and let \( U \) be
a connected component of the fiber product \( V^{\sharp} \times_{V}
V\sptilde \). Then \( q(U) = q(V\sptilde) = q^{\max}(V) \), and
hence \( U \to V^{\sharp} \) factors through the Albanese closure
\((V^{\sharp})\sptilde \to V^{\sharp} \) of \( V^{\sharp} \). On the
other hand, \( (V^{\sharp})\sptilde \) is birational to \( V\sptilde
\) by the birational map \( V \ratmap V^{\sharp} \). Hence we have \( U =
V^{\sharp} \times_{V} V\sptilde \isom (V^{\sharp})\sptilde \) by comparing
the mapping degrees.
So there is a nearly \'etale rational map
\( h\sptilde \colon V\sptilde \ratmap V\sptilde \) such that
\( \delta \circ h\sptilde = h \circ \delta\).

Let \( \alpha \colon V\sptilde \to A \) be the Albanese map of
\(V\sptilde \). Then \( h\sptilde \) induces an \'etale endomorphism
\( h_{A} \colon A \to A \) such that  \( \alpha \circ h\sptilde =
h_{A} \circ \alpha \). By \cite[Theorem 8.3, Corollary 8.4]{Ka85},
there are an isogeny $A' \to A$ of abelian varieties
and an isomorphism \( V\sptilde \times_{A} A' \isom F \times A'\) over \(A'\)
for a normal projective variety \( F \) with only canonical singularities and
\(q(F) = 0\).
By replacing the isogeny \( A' \to A\), we may assume that \(A' = A \) and
\( A' \to A \) is the
multiplication map by \( m > 0 \). By Lemma~\ref{lem:mu} below, there
is an \'etale endomorphism \( \varphi_{A} \) of \( A \) such that
\(h_{A}(ma) = m\varphi_{A}(a) \) for any \( a \in A \). So
there is a nearly \'etale rational map \( \varphi \colon F \times A
\ratmap F \times A \) such that \( p_{2} \circ \varphi = \varphi_{A}
\circ p_{2} \) and \( \theta \circ \varphi = h\sptilde \circ
\theta \), where \( \theta \colon F \times A \to V\sptilde \) is the
composite of the isomorphism \( F \times A' \isom V\sptilde
\times_{A} A' \) and the projection
\( V \sptilde \times_{A} A' \to V\sptilde \).
Then \( \varphi = \varphi_{F}
\times \varphi_{A} \) for a nearly \'etale rational map
\(\varphi_{F} \colon F \ratmap F \), by Proposition~\ref{prop:bir diagonal}.
Applying the same argument above to
\(\varphi_F \colon F \ratmap F\) instead of
\(h \colon V \ratmap V\),
we complete the proof of Theorem~B by induction on \(\dim V\).
\hfill \qedsymbol

\vspace{2ex}

\begin{lem}\label{lem:mu}
Let \( \mu \colon A \to A \) be the multiplication map \( a \mapsto
ma \) by a positive integer \( m \) for an abelian variety \( A \)
with a given abelian group structure. Then,
for a morphism \( h \colon A \to A \),
there is a morphism \( h' \colon A \to A \) such that
\( \mu \circ h' = h \circ \mu\).
\end{lem}

\begin{proof}
There exist a homomorphism \( \varphi \colon A \to A \) of abelian
group and a point \( c \in A \) such that \( h(a) = \varphi(a) + c \)
for any \( a \in A \). There is a point \( c' \in A \) with \(mc' = c \),
since \( A \) is divisible. We define \( h' \colon A \to A\) by
\( h'(a) = \varphi(a) + c' \) for \( a \in A \). Then
\[ mh'(a) = m\varphi(a) + mc' = \varphi(ma) + c = h(ma). \qedhere\]
\end{proof}

\subsection{Conjectural discussion}
\label{subsect:conj_discuss}

We shall pose the following:

\begin{conj}\label{conj:Bo}
Let \( X \) be a nonsingular projective variety such that
\( \kappa(X) = q^{\max}(X) = 0 \). Then \( \pi_{1}(X) \) is finite.
\end{conj}

This is true for \( X \) with \( K_{X} \) numerically trivial, by
Bogomolov's decomposition theorem.
Furthermore, it is true when \( \dim X \leq 3 \) (cf.\ \cite{NS}).

\begin{lem}
Let \( V \) be a normal projective variety such that
\( \pi_{1}^{\alg}(M) \) is finite for
a resolution \( M \to V \) of singularities of \(V\).
Then any nearly \'etale rational endomorphism
\( h \colon V \ratmap V \) is birational.
\end{lem}

\begin{proof}
We may assume that \(V\) is nonsingular.
For an arbitrary positive integer \(l\), let
\( V \ratmap V^{\sharp (l)} \to V \) be the Stein factorization
of the \( l \)-th power \( h^{l} = h \circ \cdots \circ h \).
Then \( V^{\sharp (l)} \to V \) is an \'etale morphism of degree
\( (\deg h)^{l} \) by Lemma~\ref{lem:nearlyEt vs pi1}.
Now, \( \pi_{1}^{\alg}(V) \) is finite, since
\( \pi_{1}(M) \isom \pi_{1}(V) \).
Hence \( \deg h = 1 \).
\end{proof}

So, if Conjecture~\ref{conj:Bo} is true, then a nearly \'etale
rational endomorphism of a weak Calabi-Yau variety is birational.
In particular, the\emph{ building blocks} of the \'etale endomorphisms of
projective varieties with \( \kappa = 0 \) would then turn out to be
the endomorphisms of abelian varieties and the birational automorphisms
of weak Calabi-Yau varieties.

\section{The uniruled case}
\label{sect:mrc}

\subsection{Maximal rationally connected fibration}
\label{subsect:pre ThC}

A projective variety \(X\) is called \emph{uniruled} if there is
a dominant rational map \( \BPP^{1} \times Y \ratmap X \) for a
variety \( Y \) with \( \dim Y = \dim X - 1 \).
For a nonsingular projective variety \( X \), the following
three conditions are all equivalent (cf.\ \cite[\S3]{Cp91}, \cite[\S 2]{KoMM}):
\begin{enumerate}
\item  Any two points of \( X \) are connected by a chain of rational
curves.

\item  Any two general points of \( X \) are contained in one
and the same rational curve.

\item  There is a nonsingular rational curve on \( X \)
with ample normal bundle.
\end{enumerate}
If one of the conditions above is satisfied,
then \( X \) is called \emph{rationally connected}.

\begin{rem}[{cf.\ \cite[\S3]{Cp91}, \cite[Proposition (2.5)]{KoMM}}]%
\label{rem:ratCon}
A nonsingular rationally connected variety \( X \) has the
following properties:
\begin{itemize}
\item  \( X \) is simply connected.

\item  \( \OH^{i}(X, \SO_{X}) = 0 \) for \( i > 0 \).

\item  \( \OH^{0}(X, (\Omega_{X}^{1})^{\otimes m}) = 0 \) for any
\( m > 0 \).
\end{itemize}
\end{rem}

If \(X\) is a uniruled nonsingular
projective variety, then there exists uniquely up to birational
equivalence, a rational fiber space
\( \pi \colon X \ratmap Y \) over a nonsingular projective variety
\( Y \) satisfying the following conditions,
by \cite{Cp}, \cite{KoMM}, \cite{GHS}:
\begin{enumerate}
\item \( \pi \) is weakly holomorphic, i.e.,
there exists open dense subsets \( U \subset X\), \( V \subset Y \)
such that \( \pi \) induces a proper surjective morphism \( U \to V \).

\item  For a general point \( P \) of the open set \( U \) above,
the fiber over \( \pi(P) \) is a maximal rationally connected
submanifold of \( X \) containing \( P \).

\item  \(Y\) is not uniruled (cf.\ \cite[Corollary 1.4]{GHS}).
\end{enumerate}
The fibration \( \pi \) is called the
\emph{maximal rationally connected fibration} of \( X \).

\begin{lem}\label{lem:1}
Let \( f \colon X \to X\) be an \'etale endomorphism of a uniruled
projective variety \( X \).
Then there exist a proper birational
morphism \( \mu \colon M \to X \), a proper surjective morphism
\( \pi \colon M \to Y \), an \'etale endomorphism \( f_{M} \) of \( M \),
and an endomorphism \( h \) of \( Y \) satisfying the following conditions\emph{:}
\begin{enumerate}
\item  \( M \) is a nonsingular projective variety.
\item \( Y \) is a normal and non-uniruled projective variety.
\item \( \pi \) is birational to the maximal rationally connected
fibration of \( M \).
\item  \( \mu \circ f_{M} = f \circ \mu \) and \( \pi \circ f_{M} = h \circ \pi \)
\emph{(cf.\ the diagram in Theorem~C)}.
\end{enumerate}
\end{lem}

\begin{proof}
We may assume that \( X \) is nonsingular, by replacing it with
an equivariant resolution of singularities with respect to \( f \)
(cf.\ Section~\ref{subsect:equiv resol}).
Let \( \Psi \colon X \ratmap \Chow(X) \) be the rational map to the Chow variety of
\( X \) which defines the maximal rationally connected fibration of \(X\).
By associating a cycle \( Z  \)
of \( X \) with the push-forward \( f_{*}Z \), we have a functorial
morphism \( f_{c} \colon \Chow(X) \to \Chow(X) \) since \(f\) is finite
(cf.\ \cite[Chapter IV \S 2 Th\'eor\`eme 6]{Ba},
\cite[Th\'eor\`eme 6.3.1]{An}, \cite[Theorem I.6.8]{Ko96}).
Since \(f\) is \'etale, \(f_*Z\) is reduced for a rationally connected submanifold \(Z\).
Thus, \(\Psi \circ f = f_c \circ \Psi\).
Let \( X \ratmap Y \to \Chow(X) \) be the
Stein factorization of \(\Psi\) and \( h \colon Y \to Y \) the
endomorphism induced from \( f_{c} \) and \( f \).
Note that \( Y \) is not uniruled.
An endomorphism of the graph
\(\Gamma \subset X \times Y \) of the rational map \( X \ratmap Y \)
is induced from \( f \times h \).
Let \( T \to \Gamma \) be the normalization, \( f_{T} \) the induced
endomorphism of \( T \), and
\(p_1 \colon T \to X\) the morphism
induced from the first projection \(\Gamma \subset X \times T \to T\).
Then, \(p_1 \circ f_T = f \circ p_1\). Thus, \(p_1\) factors
through a finite birational morphism \(T \to T_1\) to
the fiber product \(T_1 := X \times_{X} T\) of \(f \)
and \(p_1\) over \(X\). Since \(f\) is \'etale,
\(T \isom T_1\), and hence \(f_T\) is \'etale.
Let \( M \to T \) be an equivariant resolution of singularities with
respect to \( f_{T} \), and \( f_{M}  \) the lift of
\( f_{T} \) to \( M \) as an \'etale endomorphism.
Then, these data satisfy the required conditions.
\end{proof}

The following is proved essentially in \cite[Theorem 5.2]{Ko93},
where \( F \) is assumed to be a rationally connected manifold.
We present a slightly different proof.

\begin{lem}\label{lem:pi1}
Let \( g \colon M \to N \) be a proper surjective morphism between
nonsingular varieties. For a general fiber \( F \) of \( g \),
suppose that
\begin{enumerate}
    \item \label{lem:pi1:conn} \( F \) is connected,

    \item \label{lem:pi1:sconn} \( F \) is simply connected, and

    \item \label{lem:pi1:vanish} \( \OH^{i}(F, \SO_{F}) = 0 \) for
    any \( i > 0 \).
\end{enumerate}
Then \( g_{*} \colon \pi_{1}(M) \to \pi_{1}(N) \) is isomorphic.
\end{lem}

\begin{proof}
\emph{Step}~1. If \( N^{\circ} \subset N \) is a Zariski open subset
with the codimension of \( N \setminus N^{\circ} \) bigger than one,
then \( \pi_{1}(N^{\circ}) \isom \pi_{1}(N) \), and
\(\pi_{1}(g^{-1}(N^{\circ})) \to \pi_{1}(M) \) is surjective. Thus,
if \( \pi_{1}(g^{-1}(N^{\circ})) \to \pi_{1}(N^{\circ}) \) is
an isomorphism, then so is \( \pi_{1}(M) \to \pi_{1}(N) \). Hence, we
may replace \( N \) with such an open subset \( N^{\circ} \). In
particular, we may assume that \( g \colon M \to N \) is smooth
outside a nonsingular divisor \( D = \sum D_{i} \) of \(N\), where
\( D_{i}\) is an irreducible component.

\emph{Step}~2. Let \( \SU_{i} \) be an analytic open neighborhood of
a point \( P_{i} \in D_{i} \) of \(D\) such that \( \SU_{i} \) is
biholomorphic to a unit polydisc
\( \{(t_{1},t_{2}, \ldots, t_{n}) \in \BCC^{n} \, ; \, |t_{j}| < 1
\text{ for any } j \} \)
in which \( P_{i} \) is mapped to the origin
\( 0 = (0, 0, \ldots, 0) \) and \( \SU_{i} \cap D \) is mapped to
the coordinate hypersurface \( \{t_{1} = 0\} \). Since \( \SU_{i}
\setminus D = \SU_{i} \setminus D_{i}\) is homotopic to a circle, we
have a generator \( \delta_{i} \) of \( \pi_{1}(\SU_{i} \setminus D)
\). By van Kampen's theorem, or other topological argument, we infer
that the kernel of the surjection \( \pi_{1}(N \setminus D) \to
\pi_{1}(N) \) is generated by the conjugacy classes of the images
\( \overline{\delta}_{i} \) of \( \delta_{i} \) under the
homomorphisms
\( \pi_{1}(\SU_{i} \setminus D) \to \pi_{1}(N \setminus D) \).

\emph{Step}~3. By the assumptions \eqref{lem:pi1:conn} and \eqref{lem:pi1:sconn},
and by the homotopy exact sequence, we infer that the natural
homomorphism \( \pi_{1}(M \setminus g^{-1}D) \to \pi_{1}(N \setminus
D) \) is an isomorphism. Let \( \hat{\delta}_{i} \in \pi_{1}(M
\setminus g^{-1}D)\) be the element corresponding to
\(\overline{\delta}_{i} \in \pi_{1}(N \setminus D)\). In order to show
the homomorphism \( \pi_{1}(M) \to \pi_{1}(N) \) to be isomorphic,
it is enough to show that \( \hat{\delta}_{i} \) is contained in the kernel of
\(\pi_{1}(M \setminus g^{-1}D) \to \pi_{1}(M) \). Let \( C_{i}
\subset \SU_{i} \) be a curve corresponding to an axis with respect to
the coordinate system \( (t_{1}, \ldots, t_{n}) \) of the polydisc
such that \( C_{i} \cap D = \{P_i\} \) and let
\(X_{i} \) be the fiber product \( M \times_{N} C_{i} \). By changing
\( P_{i} \), \( \SU_{i} \), and coordinates \( (t_{1}, \ldots,
t_{n}) \) slightly, we may assume that \( X_{i} \) is nonsingular.
Then \( \hat{\delta}_{i} \) comes from \( \pi_{1}(X_{i} \setminus
g^{-1}(P_{i})) \isom \pi_{1}(C_{i} \setminus \{P_{i}\}) = \BZZ
\delta_{i}\). Thus, we have only to show that \( \pi_{1}(X_{i}) = 0
\), or equivalently, \( \pi_{1}(g^{-1}(P_{i})) = 0 \), since
\(g^{-1}(P_{i}) \) is a deformation retract of \( X_{i} \).

\emph{Step}~4. By \emph{Step}~3, the proof of Lemma~\ref{lem:pi1} is
reduced to the case where \( N \) is a unit disc and \( g \) is
smooth outside the origin \( 0 \). We shall show
\(\pi_{1}(g^{-1}(0)) = 0 \) in this case. By shrinking \( N \) if
necessary, we have a holomorphic curve \( T \subset M \) such that
\( T \) is biholomorphic to a unit disc and \( T \to N \) is
a finite surjective morphism branched only at \( 0 \in N \). We have:
\begin{itemize}
\item  \( \pi_{1}(M \setminus g^{-1}(0)) \isom \pi_{1}(N
\setminus \{0\}) \isom \BZZ\),

\item  \( \pi_{1}(T \setminus g^{-1}(0)) \to
\pi_{1}(N \setminus \{0\})\) is an injection into a finite-index
subgroup, and
\item \( \pi_{1}(M \setminus g^{-1}(0)) \to \pi_{1}(M) \)
is surjective.
\end{itemize}
Therefore, \( \pi_{1}(M) \) is a finite cyclic group.
Let \( \lambda \colon \widetilde{M} \to M \) be the universal
covering map and let
\( \widetilde{M} \to \widetilde{N} \to N \)
be the Stein factorization of
\( g \circ \lambda \colon \widetilde{M} \to M \to N \).
Then the induced morphism
\(\widetilde{M} \to M \times_{N} \widetilde{N} \) is isomorphic over
\( (N \setminus \{0\}) \times_{N} \widetilde{N}  \) by
\eqref{lem:pi1:sconn}.
Applying Lemma~\ref{lem:simpconn} to the commutative diagram
corresponding to the Stein factorization of \(g \circ \lambda\),
we infer that \(\widetilde{N} \to N \) is \'etale.
In fact, the condition \eqref{lem:simpconn:cond5}
of Lemma~\ref{lem:simpconn} over \( N \) and
\( \widetilde{N} \) are both satisfied, since the higher
direct image sheaves \( \OR^{i}g_{*}\SO_{M} \)
are torsion free over the nonsingular curves \( N \) and
the same thing holds for the higher direct image sheaves for
\(\widetilde{M} \to \widetilde{N}\).
Hence, \( \widetilde{N} \isom N \)
and \( \widetilde{M} \isom M \). Therefore, \( M \) and \(g^{-1}(0) \)
are simply connected.
\end{proof}

The following gives a sufficient condition for a finite morphism
to be nearly \'etale.

\begin{prop}\label{prop:2}
Let \( h \colon \widetilde{Y} \to Y \) be a finite surjective
morphism between normal varieties with \( \deg h > 1 \).
Then \( h \) is nearly \'etale
if there exist proper surjective morphisms \( \pi \colon M \to Y \)
and  \( \tilde{\pi} \colon \widetilde{M} \to \widetilde{Y} \),
and a finite \'etale covering \( f \colon
\widetilde{M} \to M \) satisfying
the following conditions\emph{:}
\begin{enumerate}
\item  \( M \) and \( \widetilde{M} \) are nonsingular, and
\( \pi \circ f = h \circ \tilde{\pi} \).

\item A general fiber \( F \) of \( \pi \)
is connected and simply connected with \( \OH^{i}(F, \SO_{F})
= 0 \) for any \( i > 0 \).
\end{enumerate}
\end{prop}

\begin{rem}\label{rem:normalization}
Let \( M_{1} \) be the fiber product
\( \widetilde{Y} \times_{Y} M\).
Then \( f \) induces a finite morphism
\( \lambda \colon \widetilde{M} \to M_{1} \) over \( M \)
with the commutative diagram below:
\[ \begin{CD}
\widetilde{M} @>{\lambda}>> M_{1} @>>> M \\
@V{\tilde{\pi}}VV @VVV @V{\pi}VV\\
\widetilde{Y} @= \widetilde{Y} @>h>> \rule[-1.5ex]{0ex}{1ex} Y .
\end{CD} \]
\noindent
We can show that  \( \lambda \) is just the normalization of the
reduced structure \( M_{1, \red} \) of \( M_{1} \), as follows:
Note that \( M_{1} \) is irreducible since \( \pi \) has only
connected fibers and \( h \) is finite.
Since a general fiber of \( \pi \) is simply connected,
\(\lambda \) is an isomorphism over \( h^{-1}(U) \times_{Y} M \subset
\widetilde{Y} \times_{Y} M = M_{1} \) for a dense open subset \( U \subset Y \).
Then \(\lambda \) induces a finite and birational morphism
from \(\widetilde{M}\) to the normalization of \(M_{1, \red}\),
which is turned to be an isomorphism.
\end{rem}

\begin{proof}[Proof of \emph{Proposition~\ref{prop:2}}]
Let \( \nu \colon N \to Y \) be a resolution
of singularities and \( M' \to M \) a proper birational
morphism from a nonsingular variety \( M' \) such that \( M' \to M
\ratmap N \) is a morphism.
Then, \(\pi_{1}(M') \isom \pi_{1}(N) \) by Lemma~\ref{lem:pi1}.
Let \( \Pi \subset \pi_{1}(N) \) be the image
of \( f_{*}(\pi_{1}(\widetilde{M})) \subset \pi_{1}(M) \) via
\( \pi_{1}(M) \isom \pi_{1}(M') \isom \pi_{1}(N) \), and
\( \phi \colon \widetilde{N} \to N \) the finite \'etale
covering corresponding to the finite-index
subgroup \( \Pi \).
Let
\(\widetilde{M'} \to M'\) be the \'etale covering obtained as the
pullback of \( f \) by \( M' \to M \). Then we have a commutative
diagram
\[ \begin{CD}
\widetilde{M} @<<< \widetilde{M'} @>>> \widetilde{N} \\
@V{f}VV @VVV @V{\phi}VV\\
M @<<< M' @>>> \rule[-2ex]{0ex}{1ex}\phantom{.}N \phantom{.}
\end{CD} \]
in which each square is Cartesian.
By considering the Stein factorization of \( \nu \circ \phi \), we
have a proper birational morphism \( \tilde{\nu} \colon
\widetilde{N} \to \widetilde{Y} \) such that \( h \circ \tilde{\nu}
= \nu \circ \phi\). Thus, the diagram
\[ \begin{CD}
\widetilde{M} @>{\tilde{\pi}}>>\widetilde{Y} @<{\tilde{\nu}}<< \widetilde{N} \\
@V{f}VV @V{h}VV @V{\phi}VV\\
M @>{\pi}>> Y @<{\nu}<< \rule[-2ex]{0ex}{1ex}N
\end{CD} \]
is also commutative. Therefore, \( h \) is nearly \'etale.
\end{proof}

\subsection{Proof of Theorem C}
\label{subsect:pf ThC}
We apply Lemma~\ref{lem:1} to the given \'etale endomorphism \( f \)
of \( X \).
Let \( \mu \colon M \to X \), \( \pi \colon M \to Y \), \( f_{M} \),
and \( h \) be the same objects as in Lemma~\ref{lem:1}.
By Proposition~\ref{prop:2},
\( h \) is nearly \'etale. This completes the proof of Theorem~C.
\hfill \qedsymbol

\appendix
\renewcommand{\thethm}{\Alph{section}.\arabic{thm}}
\renewcommand{\theequation}{\Alph{section}--\arabic{equation}}

\section{Topological entropies and fiber spaces}
\label{append:ent}

The reduction (A) in the introduction is deduced from Theorem~A.
We have also a similar reduction of type (A) in the dynamical
study of endomorphisms. More precisely,
for the study of the first dynamical degree and the
topological entropies of holomorphic surjective
endomorphisms of compact K\"ahler manifolds,
the case of positive Kodaira dimension is reduced to the case of
zero Kodaira dimension by
Theorem~A, and by Theorem~D below.
The reduction of non-holomorphic meromorphic endomorphisms
is not treated here.

\begin{thmD}
Let \( \pi \colon X \to Y \) be a proper surjective morphism from a compact
K\"ahler manifold \( X \) to a compact complex
analytic variety \( Y \) and
\(f \colon X \to X \) a surjective endomorphism such that
a general fiber of \( \pi \) is connected and
\( \pi \circ f = \pi \). Let \( F \) be a smooth fiber of \( \pi \),
i.e., a fiber along which \( \pi \) is smooth. Then\emph{:}
\begin{enumerate}
\item  \label{ThD:1} The equality
\( d_{1}(f) = d_{1}(f|_{F})\) holds for the first dynamical degrees \(d_1\) of
\( f \colon X \to X \) and the restriction \( f|_{F} \colon F \to F \) of \( f \).

\item  \label{ThD:2} If \( f \) is \'etale, then the equality
\( h_{\topo}(f) = h_{\topo}(f|_{F}) \) holds for the topological
entropy \( h_{\topo} \).
\end{enumerate}
\end{thmD}

The proof is based on basic properties of spectral radii, a
generalized notion of K\"ahler cone, a simple calculation of dynamical
degrees, and the results of Gromov \cite{G} and Yomdin \cite{Y} on
topological entropies.
The proof of the first assertion \eqref{ThD:1} is given at the end of
Section~\ref{subsect:dynamicaldegree}
and the proof of the second assertion \eqref{ThD:2} at the end of
Section~\ref{subsect:topent}. Some well-known properties on dynamical degrees
are prepared in Sections \ref{subsect:sprad} and \ref{subsect:dynamicaldegree}.

\subsection{Spectral radii and K\"ahler cones}
\label{subsect:sprad}

We recall some basic properties of spectral radii,
especially a generalization of the Perron-Frobenius theorem.
Furthermore, we introduce a notion of K\"ahler \( (k, k) \)-forms and the
K\"ahler cones in \( \OH^{k, k}(M, \BRR) \) for a compact K\"ahler
manifold \( M \).

The spectral radius \( \rho(\varphi) = \rho(V, \varphi) \)
of an endomorphism \( \varphi \colon V \to V \) of a
finite-dimensional \( \BCC \)-vector space \( V \) is defined to be
the maximum of the absolute values of the eigenvalues of \( \varphi \).
The spectral radius of an endomorphism
\( \varphi_{\BRR} \colon V_{\BRR} \to V_{\BRR} \) of a
finite-dimensional real vector space \( V_{\BRR} \) is defined as that
of its complexification
\( \varphi_{\BRR} \otimes_{\BRR} \BCC \colon V_{\BRR} \otimes_{\BRR} \BCC
\to V_{\BRR} \otimes_{\BRR} \BCC \).

\begin{remn}\label{lem:norm_spr}
Let \( \norm{\cdot} \) be any norm of \( V \) and let
\( \norm{\cdot}_{1} \) be the \( L^{1} \)-norm of
\( \operatorname{End}_{\BCC}(V) \) defined by
\( \norm{\varphi}_{1} : = \sup \{\norm{\varphi(v)} \, ; \, \norm{v} = 1\} \).
Then
\( \rho(\varphi) = \lim\nolimits_{m \to \infty}
\left(\norm{\varphi^{m}}_{1}\right)^{1/m}\).
\end{remn}

\begin{nottn}
Let \( V_{\BRR} \) be a finite-dimensional real vector space.
A subset \( \SC \subset V_{\BRR} \) is called a \emph{convex cone}
if \( \SC + \SC \subset \SC \) and \( \BRR_{+}\SC \subset \SC \),
where \( \BRR_{+} \) denotes the set of positive real numbers.
If \( \SC \cap (-\SC) \subset \{0\} \) in addition, then \( \SC \) is
called \emph{strictly convex}.
\end{nottn}

\begin{remn}
A convex cone \( \SC \subset V_{\BRR} \) is strictly convex if and
only if there exists a linear form \( \chi \colon V_{\BRR} \to \BRR
\) such that \( \chi > 0 \) on \( \SC \setminus \{0\} \).
\end{remn}

The following is known as a generalization of the
Perron-Frobenius theorem on real \( n \times n \)
matrices \( A = (a_{ij}) \) of positive entries \( a_{ij} \):

\begin{thm}[{cf.\ \cite{Bi}}]\label{thm:Birkhoff}
Let \( \SC \) be a strictly convex closed
cone of a finite-dimensional real vector space \( V_{\BRR} \)
such that \( \SC \) generates \( V_{\BRR} \) as a vector space.
Let \( \varphi \colon V_{\BRR} \to V_{\BRR} \) be an endomorphism
such that \( \varphi(\SC) \subset \SC \).
Then the spectral radius \( \rho(\varphi) \) is
an eigenvalue of \( \varphi \) and there is an eigenvector in \( \SC \)
with the eigenvalue \( \rho(\varphi) \).
\end{thm}

We add another property of spectral radius which is proved by using
Theorem~\ref{thm:Birkhoff}:

\begin{prop}\label{prop:B}
Let \( \varphi \colon V_{\BRR} \to V_{\BRR}\) be an endomorphism
of a finite-dimensional real vector space \(V_{\BRR}\).
Suppose that \( \varphi(\SC) \subset \SC \) for a strictly convex
closed cone \( \SC \subset V_{\BRR} \) with \( \SC + (-\SC) = V_{\BRR} \).
Let \( \norm{\cdot} \) be a norm on \( V_{\BRR} \) and
let \( \chi \colon V_{\BRR} \to \BRR \) be a linear form such that
\( \chi > 0 \) on \( \SC \setminus \{0\}\).
If \( u \) is contained in the interior of \( \SC \), then
\[ \rho(\varphi) = \lim\nolimits_{m \to \infty}
\norm{\varphi^{m}(u)}^{1/m} = \lim\nolimits_{m \to \infty}
\chi(\varphi^{m}(u))^{1/m}.\]
\end{prop}

\begin{proof}
We define \( \norm{\cdot}_{u} \colon V_{\BRR} \to \BRR_{\geq 0}\) by
\( \norm{x}_{u} := \inf\{ r \geq 0 \mid x + ru \in \SC,
-x + ru \in \SC\}\).
Then \( \norm{\cdot}_{u} \) is a norm of \( V_{\BRR} \).
There exist positive constants \( C_{1} \), \( C_{2} \) such that
\( C_{1}\norm{x} \leq \norm{x}_{u} \leq C_{2}\norm{x} \)
for any \( x \in V \). Hence, we may assume that \( \norm{\cdot} =
\norm{\cdot}_{u} \).
For the induced \( L^{1} \)-norm \( \norm{\cdot}_{u, 1} \), we have
\( \norm{\varphi^{m}}_{u, 1} \geq \norm{\varphi^{m}(u)}_{u} \) for
\( m \geq 1 \), since
\( \norm{u}_{u} = 1 \). Applying \( \chi \) to
\( \norm{\varphi^{m}(u)}_{u}u - \varphi^{m}(u) \in \SC \), we have
\( \norm{\varphi^{m}(u)}_{u}\chi(u) \geq \chi(\varphi^{m}(u))\).
We have an eigenvector \( v \in \SC \) with the eigenvalue \(\rho(\varphi)\)
by Theorem~\ref{thm:Birkhoff}.
Applying \( \chi \circ \varphi^{m} \) to
\( \norm{v}_{u} u - v \in \SC \), we have
\(\norm{v}_{u} \chi(\varphi^{m}(u)) \geq \rho(\varphi)^{m}\chi(v)\).
Therefore,
\[ \norm{\varphi^{m}}_{u, 1} \geq \norm{\varphi^{m}(u)}_{u} \geq
\frac{\chi(\varphi^{m}(u))}{\chi(u)} \geq
\frac{\chi(v)}{\chi(u)\norm{v}_{u}} \rho(\varphi)^{m}. \]
Hence, we complete the proof by
\[ \rho(\varphi) = \lim\nolimits_{m \to \infty}
\norm{\varphi^{m}}_{u, 1}^{1/m} \geq
\lim\nolimits_{m \to \infty} \norm{\varphi^{m}(u)}_{u}^{1/m}
\geq \lim\nolimits_{m \to \infty} \chi(\varphi^{m}(u))^{1/m}
\geq \rho(\varphi). \qedhere\]
\end{proof}

Let \( M \) be a compact K\"ahler manifold of dimension \( n \).
Let \( \omega \) be a \( C^{\infty} \)-\( (k, k) \)-form on \( M \)
for an integer \( 1 \leq k \leq n \).
For a local coordinate system \( (z_{1}, z_{2}, \ldots, z_{n}) \) of \( M \),
\( \omega \) is locally expressed as
\[ \omega = (\sqrt{-1})^{k^{2}}
\sum\nolimits_{I, J \subset \{1, 2, \ldots, n\}}
a_{I, J} d z_{I} \wedge d \bar{z}_{J} \]
for \( C^{\infty} \)-functions \( a_{I, J} \), where
\( \sharp I = \sharp J = k \) and
\[ d z_{I} := d z_{i_{1}} \wedge d z_{i_{2}} \wedge \cdots \wedge d
z_{i_{k}} \]
when \( I = \{ i_{1}, i_{2}, \ldots, i_{k}\} \) with \( i_{1} <
i_{2} < \cdots < i_{k} \).
The \( \omega \) is called a \emph{K\"ahler} \( (k, k) \)-form if
\( \omega \) is \( d \)-closed (\( d \omega = 0 \)) and
real (\( \overline{\omega} = \omega \)), and if the matrix
\( (a_{I, J}) \) is positive-definite everywhere in \( M \).
Note that this is just the usual definition of K\"ahler
forms in case \( k = 1 \). The following is easily shown:

\begin{lem}\label{lem:Kaehler}\hfill
\begin{enumerate}
\item \label{lem:Kaehler:item1}
If \( \eta \) is a \emph{(}usual\emph{)} K\"ahler form
\emph{(}\( (1, 1) \)-form\emph{)},
then \( \eta^{k} = \eta \wedge \cdots \wedge \eta \) is
a K\"ahler \( (k, k) \)-form for \( 1 \leq k \leq n \).

\item \label{lem:Kaehler:item2}
Let \( \omega \) be a K\"ahler \( (k, k) \)-form.
If \( \omega' \) is a K\"ahler \( (n - k, n - k) \)-form, then
\( \int_{M}\omega \wedge \omega' > 0\).
If \( T \) is a \( d \)-closed positive \( (n-k, n-k) \)-current, then
\( \int_{M} \omega \wedge T \geq 0 \).
\end{enumerate}
\end{lem}

Inside of the real vector space
\( \OH^{k, k}(M, \BRR) := \OH^{k, k}(M) \cap \OH^{2k}(M, \BRR) \),
the set \( P^{k}(M) \) of the classes \( [\omega] \) of
K\"ahler \( (k, k) \)-forms \( \omega \) on \( M \)
is a strictly convex open cone. It is called the
\emph{K\"ahler cone} of degree \( k \).
Its closure in \( \OH^{k, k}(M, \BRR) \) is
denoted by \( \overline{P^{k}(M)} \).

\begin{lem}\label{lem:positive}
\begin{enumerate}
\item \label{lem:positive:1}
\( \int_{M} \xi \cup [\omega] > 0\) for
any \( \xi \in \overline{P^{k}(M)} \setminus \{0\} \) and
for any K\"ahler \( (n - k, n - k) \)-form \( \omega \).

\item \label{lem:positive:2}
If \( \theta \in \overline{P^{1}(M)} \) and
if \( \theta - [T] \in P^{1}(M) \) for a \( d \)-closed positive
\( (1, 1) \)-current \( T \), then for any \( 1 \leq l \leq n \),
there exists an element \( z \in \overline{P^{l-1}(M)} \) such that
\( \theta^{l} - z \cup [T] \in P^{l}(M)  \) and \( z \cup [T] \)
is represented by a \( d \)-closed positive \( (l, l) \)-current.
In particular, \( \int_{M} \xi \cup \theta^{n - k} > 0 \)
for any \( \xi \in \overline{P^{k}(M)} \setminus \{0\} \).
\end{enumerate}
\end{lem}

\begin{proof}
\eqref{lem:positive:1}:
Let \( (x, y) \) denote \( \int_{M} x \cup y \) for
\( x \in \OH^{k, k}(M, \BRR)  \) and
\( y \in \OH^{n-k, n - k}(M, \BRR) \). Then
\( \OH^{n-k, n - k}(M, \BRR) \) is dual to \( \OH^{k, k}(M, \BRR) \) by
\( (*, *) \).
Since \( P^{n - k}(M) \) generates \( \OH^{n-k, n - k}(M, \BRR)  \) as
the vector space, we can find a K\"ahler \( (n - k, n - k) \)-form
\( \omega_{0} \) such that \( (\xi, [\omega_{0}]) \ne 0 \).
By Lemma~\ref{lem:Kaehler}, we have \( (\xi, [\omega_{0}]) > 0 \).
There is a positive constant \( C \)
such that \( C \omega - \omega_{0} \) is also a
K\"ahler \( (k, k) \)-form, since \( M \) is compact. Thus, by
Lemma~\ref{lem:Kaehler},
\[ (\xi, [\omega]) \geq C^{-1}(\xi, [\omega_{0}]) > 0.\]

\eqref{lem:positive:2}:
We set \( \alpha = \theta - [T] \in P^{1}(M)\).
Then, we have
\( \theta^{l} = \alpha^{l} + z \cup [T] \),
where \( z = 1 \in \OH^{0}(M, \BRR)\) when \( l = 1 \) and
\begin{equation}\label{eq:z}
z = \sum\nolimits_{i = 0}^{l - 1} \theta^{i} \cup \alpha^{l - 1 - i}
\end{equation}
when \( l \geq 2 \). Thus, \( z \in \overline{P^{l - 1}(M)} \) and
\( z \cup [T] \) is represented by a positive current.
In particular,
\( (\xi, \theta^{n - k}) \geq (\xi, \alpha^{n - k}) > 0 \)
by Lemma~\ref{lem:Kaehler} and \eqref{lem:positive:1}.
\end{proof}

The following is a property of K\"ahler \((1,1)\)-forms.

\begin{lem}\label{lem:genfin}
Let \( \psi \colon M \to X \) be a generically finite morphism into
another compact K\"ahler manifold \( X \).
If \( \eta \) is a K\"ahler form on \( X \), then
\( [\psi^{*}\eta] - [T] \in P^{1}(M) \) for
a \( d \)-closed real positive \( (1, 1) \)-current \( T \) on \( M \).
\end{lem}

\begin{proof}
Let \( M \to V \to X \) be the Stein factorization of \(\psi\).
By Hironaka's blowing up,
we have a projective bimeromorphic morphism
\( \nu \colon Z \to V \) from another compact K\"ahler manifold \( Z \)
such that \( \mu \colon Z \to V \ratmap M \) is holomorphic.
Here, \( \SO_{Z}(-E) \) is \( \nu \)-ample for
a \( \nu \)-exceptional effective divisor \( E \).
Then \( \SO_{Z}(-E) \) is also relatively ample over \( X \), and
\( [\mu^{*}\psi^{*}\eta] - \ep[E] \) is
represented by a K\"ahler form on \( Z \) for some \( \ep > 0 \)
(cf.\ \cite[Lemma 4.4]{Fujiki_C}, \cite[Lemma 2]{Fub}).
Thus
\[ [\mu^{*}\psi^{*}\eta] - \ep[E] - [\mu^{*}\xi] \in P^{1}(Z) \]
for a K\"ahler form  \( \xi \) on \( M \).
Hence, \( \psi^{*}[\eta] - [\xi] \) is represented by a
\( d \)-closed real positive \( (1, 1) \)-current on \( M \).
\end{proof}

\subsection{Dynamical degrees}
\label{subsect:dynamicaldegree}

Let \( f \colon M \to M \) be a surjective endomorphism
of a compact K\"ahler manifold \( M \) of dimension \( n \).
Then \( f \) induces natural homomorphisms
\[ f^{*} \colon \OH^{i}(M, \BZZ) \to \OH^{i}(M, \BZZ) \quad
\text{and} \quad f_{*} \colon \OH_{i}(M, \BZZ) \to \OH_{i}(M, \BZZ) \]
for \( 0 \leq i \leq 2n \).
The composite
\[ \OH^{i}(M, \BZZ) \xrightarrow{f^{*}} \OH^{i}(M, \BZZ) \isom
\OH_{2n - i}(M, \BZZ) \xrightarrow{f_{*}} \OH_{2n - i}(M, \BZZ) \isom
\OH^{i}(M, \BZZ) \]
is just the multiplication map by \( \deg f \), where the isomorphism
above is induced from the Poincar\'e duality.
Thus \( f^{*} \colon \OH^{i}(M, \BRR) \to \OH^{i}(M, \BRR) \) is
isomorphic. Moreover, \( f^{*} \) preserves the Hodge structure,
i.e., \( f^{*}\OH^{p, q}(M) = \OH^{p, q}(M) \).
Note that \( f \) is a finite morphism, since
\( f^{*} \colon \OH^{1, 1}(M, \BRR) \to \OH^{1, 1}(M, \BRR) \)
is an isomorphism.

\begin{dfn}[Dynamical degree]
Let \( \eta \) be a K\"ahler form on \( M \).
For an integer \( 1 \leq l \leq n \), we set
\[ \delta_{l}(f, \eta) :=
\int_{M} f^{*}\eta^{l} \wedge \eta^{n - l}, \]
where \( \eta^{i} \) denotes the \( i \)-th power
\( \eta \wedge \cdots \wedge \eta \) for \( 1 \leq i \leq n \),
and \( \eta^{0} := 1 \).
The \( l \)-th \emph{dynamical degree} of \( f \) is defined to be
\[ d_{l}(f) := \lim\nolimits_{m \to \infty}
\left(\delta_{l}(f^{m}, \eta)\right)^{1/m}. \]
\end{dfn}

Note that \( d_{0}(f) = 1 \) and \( d_{n}(f) = \deg f \)
by the definition.

\begin{lem}\label{lem:indep}
Let \( x \) be an element of \( \overline{P^{1}(M)}\) and
\( y \) an element of \( \overline{P^{k}(M)} \) such that
\( x - [T_{1}] \in P^{1}(M) \) and
\( y - [T_{k}] \in P^{k}(M) \)
for certain \( d \)-closed real positive currents \( T_{1} \)
of type \( (1, 1) \) and \( T_{k} \) of type \( (k, k) \)
for \( 1 \leq k \leq n \).
Then
\[ d_{n - k}(f) = \lim\nolimits_{m \to \infty}
\left(\int_{M} (f^{m})^{*}(x^{n - k}) \cup y\right)^{1/m}. \]
\end{lem}

\begin{proof}
There exist \( d \)-closed positive currents \( S_{1} \) of type
\( (1, 1) \) and \( S_{k} \) of type \( (k, k) \), and a constant \( a > 0 \)
such that
\( ax - [\eta] = [S_{1}]\) and \( ay - [\eta^{k}] = [S_{k}]\).
We set \( l := n - k \). Then
\( a^{l}x^{l} - [\eta]^{l} = z \cup [S_{1}] \) for
an element \( z \in \overline{P^{l-1}(M)} \) by
Lemma~\ref{lem:positive}, \eqref{lem:positive:2}.
Thus \( f^{*}(a^{l}x^{l} - [\eta]^{l}) =
f^{*}z \cup f^{*}[S_{1}] \) is represented by a \( d \)-closed
positive \( (l, l) \)-current, since \( f^{*}[S_{1}] = [f^{*}S_{1}] \)
for the positive \( (1, 1) \)-current \( f^{*}S_{1} \).
Therefore,
\[f^{*}(a^{l}x^{l}) \cup (ay) - f^{*}[\eta]^{l} \cup [\eta]^{k}
= f^{*}z \cup f^{*}[S_{1}] \cup (ay)
+ f^{*}[\eta]^{l} \cup [S_{k}] \]
is represented by a positive \( (n, n) \)-current. Hence, for any \(
m \geq 1 \),
\begin{equation}\label{eq:indep1}
a^{l+1}\int_{M} (f^{m})^{*}(x^{l}) \cup y \geq \int_{M} (f^{m})^{*}\eta^{l}
\wedge \eta^{k} = \delta_{l}(f^{m}, \eta).
\end{equation}
Conversely,
there exists a constant \( b > 0 \) such that
\( b[\eta] - x \in P^{1}(M)\) and
\( b[\eta^{k}] - y \in P^{k}(M)\).
Then \( b^{l}[\eta]^{l} - x^{l} \in \overline{P^{l}(M)} \)
(cf.\ \eqref{eq:z}) and hence
\( f^{*}(b^{l}[\eta]^{l} - x^{l}) \in \overline{P^{l}(M)} \).
Therefore
\[ f^{*}(b^{l+1}[\eta]^{l}) \cup [\eta]^{k} - f^{*}(x^{l}) \cup y
\in \overline{P^{n}(M)}.\]
Thus, for any \( m \geq 1 \), we have
\begin{equation}\label{eq:indep2}
b^{l+1} \delta_{l}(f^{m}, \eta) = b^{l+1}\int_{M} (f^{m})^{*}\eta^{l}
\wedge \eta^{k} \geq \int_{M} (f^{m})^{*}(x^{l}) \cup y.
\end{equation}
Hence, we have the expected equality by applying \( m \to \infty\) to
the \( m \)-th roots of the both sides of the inequalities \eqref{eq:indep1}
and \eqref{eq:indep2}
\end{proof}

\begin{lem}\label{lem:d=genfin}
Suppose that there exist a generically finite surjective morphism \(
\mu \colon M \to X \) into another compact K\"ahler manifold \( X \)
and an endomorphism \( g \colon X \to X \) satisfying
\( \mu \circ f = g \circ \mu \).
Then \( d_{l}(f) = d_{l}(g) \) for any \( l \).
\end{lem}

\begin{proof}
Let \( \xi \) be a K\"ahler form on \( X \).
By Lemma~\ref{lem:genfin},
\( [a\mu^{*}\xi - \eta] \) is
represented by a \( d \)-closed real positive \( (1, 1) \)-current on
\( M \) for a certain positive constant \( a \).
Then \( [a^{l}\mu^{*}\xi^{l} - \eta^{l}] \) is represented by
a \( d \)-closed real positive \( (l, l) \)-current
by Lemma~\ref{lem:positive}, \eqref{lem:positive:2}.
Thus, by Lemma~\ref{lem:indep},
\begin{align*}
d_{l}(f) &= \lim\nolimits_{m \to \infty}
\left(\int_{M} (f^{m})^{*}(\mu^{*}\xi^{l}) \wedge
\mu^{*}\xi^{n - l}\right)^{1/m} \\
&= \lim\nolimits_{m \to \infty}
\left( (\deg \mu) \int_{X} (g^{m})^{*}\xi^{l} \wedge \xi^{n - k}
\right)^{1/m} = d_{l}(g).  \qedhere
\end{align*}
\end{proof}

\begin{lem}[{cf.\ \cite{DS05}, \cite{Guedj06}}]\label{lem:fact:sr}
The \( k \)-th dynamical degree
\( d_{k}(f) \) equals the spectral radius of the following
endomorphism induced from \(f^{*}\)\emph{:}
\[ (f^{*})^{(k, k)} \colon \OH^{k, k}(M, \BRR) \to
\OH^{k, k}(M, \BRR). \]
\end{lem}

\begin{proof}
For a K\"ahler form \( \eta \), \( [\eta^{k}] \in \OH^{k, k}(M, \BRR)\)
is contained in \( P^{k}(M) \).
We define a linear form
\( \chi \colon \OH^{k, k}(M, \BRR) \to \BRR \) by
\( \chi(x) = \int_{M} x \cup [\eta^{n - k}] \).
Then \( \delta_{k}(f, \eta) = \chi([f^{*}\eta^{k}]) \).
We have
\( \chi > 0 \) on \( \overline{P^{k}(M)} \setminus \{0\} \) by
Lemma~\ref{lem:positive}.
Moreover,
\( (f^{*})^{(k, k)}P^{k}(M) \subset \overline{P^{k}(M)} \).
Therefore, by Proposition~\ref{prop:B},
\[ \rho((f^{*})^{(k, k)}) = \lim\nolimits_{m \to \infty}
\chi([(f^{m})^{*}\eta^{k}])^{1/m} = \lim\nolimits_{m \to \infty}
\delta_{k}(f^{m}, \eta)^{1/m} = d_{k}(f). \qedhere\]
\end{proof}

\begin{fact}\label{fact:sr}
We have \( d_{l - 1}(f)d_{l + 1}(f) \leq d_{l}(f)^{2} \)
for any \( 1 \leq l < n \) (cf.\ \cite[Proposition 1.2]{Gu}).
In particular, \( d_{n}(f) = \deg f \leq d_{1}(f)^{n} \).
\end{fact}

For an element \( x \in \OH^{k, k}(M, \BRR) \) and
for \( l \leq n - k \),
we set
\[ \delta_{l}(f, \eta; x) :=
\int_{M} [f^{*}\eta^{l} \wedge \eta^{n - k - l}] \cup x . \]
Then
\( \delta_{l}(f, \eta; [\eta^{k}]) = \delta_{l}(f, \eta) \).
Let \( C > 0 \) be a constant such that
\( C[\eta^{k}] - x \in P^{k}(M) \).
Then,
\( C \delta_{l}(f, \eta) \geq \delta_{l}(f, \eta; x) \)
for any \( f \colon M \to M \).
Hence, we have
\begin{equation}\label{eq:deltax}
d_{l}(f) \geq \varlimsup\nolimits_{m \to \infty}
\left(\delta_{l}(f^{m}, \eta; x)\right)^{1/m}.
\end{equation}

\begin{prop}\label{prop:subvar}
Let \( F \) be a compact K\"ahler manifold of dimension \( k \) and let
\( \phi \colon F \to M \) be a generically finite morphism such that
\( \phi \circ h = f \circ \phi \) for
a surjective endomorphism \( h \colon F \to F \).
Then for any \( 1 \leq l \leq k\), one has
\[ d_{l}(f) \geq d_{l}(h) \quad \text{ and } \quad
d_{l + n - k}(f) \geq \deg(f)\deg(h)^{-1} d_{l}(h) . \]
\end{prop}

\begin{proof}
Let \( G_{i} \colon \OH^{2i}(F, \BCC) \to \OH^{2(n - k + i)}(M, \BCC) \)
be the Gysin homomorphism
\[ \OH^{2i}(F, \BCC) \isom \OH_{2k - 2i}(F, \BCC) \xrightarrow{\phi_{*}}
\OH_{2k - 2i}(M, \BCC) \isom \OH^{2n -2k + 2i}(M, \BCC) \]
for \( 0 \leq i \leq k \),
where the isomorphisms in both sides are induced
from the Poincar\'e duality. Note that
\( G_{i}\OH^{p, q}(F) \subset
\OH^{p + n - k, q + n - k}(M) \).
We have the ``projection formula'' \( G_{p+q}(x \cup \phi^{*}(y)) =
G_{p}(x) \cup y \) for any \( x \in \OH^{p}(F, \BCC) \) and
\( y \in \OH^{q}(M, \BCC) \).
The element \( G_{0}(1) \in \OH^{2(n - k)}(M, \BRR)\) is just the
cohomology class \( \class(\phi_{*}F) \) of the cycle
\( \phi_{*}F = \deg(F \to \phi(F))\phi(F)\).
Thus,
\[ \int_{F} (h^{m})^{*}(\phi^{*}\eta^{l}) \wedge
\phi^{*}\eta^{k - l} =
\int_{F} \phi^{*}((f^{m})^{*}\eta^{l}) \wedge \phi^{*}\eta^{k - l}
= \int_{M} [(f^{m})^{*}\eta^{l}) \wedge \eta^{k - l}] \cup
\class(\phi_{*}F) \]
for any K\"ahler form \(\eta\) and for any \( m \geq 1 \).
Since \( \phi \) is generically finite,
\( [\phi^{*}\eta] - [T] \in P^{1}(F) \)
for a \( d \)-closed real positive \( (1, 1) \)-current \( T \)
by Lemma~\ref{lem:genfin}.
Thus, we have
\begin{align}
d_{l}(h) &= \lim\nolimits_{m \to \infty} \left(
\int_{M} [(f^{m})^{*}\eta^{l}) \wedge \eta^{k - l}] \cup
\class(\phi_{*}F) \right)^{1/m} \label{eq:prop:subvar}\\
&= \lim\nolimits_{m \to \infty}
\left(\delta_{l}(f^{m}, \eta; \class(\phi_{*}F))\right)^{1/m}
\leq d_{l}(f) \notag
\end{align}
by \eqref{lem:positive:2} of Lemma~\ref{lem:positive},
Lemma~\ref{lem:indep}, and \eqref{eq:deltax}.
From \( \phi_{*} \circ h_{*} = f_{*} \circ \phi_{*} \),
we have
\[ (\deg h)^{-1} G_{i} \circ (h^{*})^{(i, i)}  =
(\deg f)^{-1} (f^{*})^{(n - k + i, n - k + i)} \circ G_{i} \]
for any \( i\),
since \( (\deg h)^{-1} h^{*} \) is the inverse of
\[ \OH^{2i}(F, \BCC) \isom \OH_{2k - 2i}(F, \BCC) \xrightarrow{h_{*}}
\OH_{2k - 2i}(F, \BCC) \isom \OH^{2i}(F, \BCC) \]
and \( (\deg f)^{-1}f^{*} \) is the inverse of
\[ \OH^{2(n - k +i)}(M, \BCC) \isom \OH_{2k - 2i}(M, \BCC)
\xrightarrow{f_{*}}
\OH_{2k - 2i}(M, \BCC) \isom \OH^{2(n - k +i)}(M, \BCC). \]
We can find an eigenvector \( w \) of \( (h^{*})^{(l, l)} \) with the
eigenvalue \( d_{l}(h) \) from the cone \( \overline{P^{l}(F)} \) by
Theorem~\ref{thm:Birkhoff}.
Thus, \( f^{*}G_{l}(w) = \deg(f)\deg(h)^{-1}d_{l}(h)G_{l}(w)  \).
We have
\[ \int_{M} G_{l}(w) \cup [\eta]^{k - l}
= \int_{F} w \cup [\phi^{*}\eta]^{k - l} > 0 \]
for the K\"ahler form \( \eta \) above
by the projection formula for \( G_{l} \) and by Lemma~\ref{lem:positive}.
Thus, \( G_{l}(w) \ne 0 \) and
\(\deg(f)\deg(h)^{-1}d_{l}(h) \) is an eigenvalue
of \((f^{*})^{(n - k + l, n - k + l)} \).
Therefore, \( d_{n - k + l}(f) \geq \deg(f)\deg(h)^{-1}d_{l}(h)  \).
\end{proof}

Now, we are ready to prove the first assertion of Theorem~D.

\begin{proof}[Proof of \eqref{ThD:1} of \emph{Theorem~D}]
We set \( d = \dim Y > 0 \).
We have \( d_{1}(f) \geq d_{1}(f|_{F}) \) by
Proposition~\ref{prop:subvar}.
If \( d_{1}(f) \leq 1 \), then \( \deg f = d_{1}(f) = d_{1}(f|_{F}) = 1 \) by
Fact~\ref{fact:sr}.
Thus, we may assume that \( d_{1}(f) > 1 \).
Let \( v \in \overline{P^{1}(X)} \subset \OH^{1, 1}(X, \BRR) \) be an
eigenvector of \( f^{*} \) with the eigenvalue \( d_{1}(f) \).
If \( v|_{F} \in \OH^{1, 1}(F, \BRR) \) is not zero, then
\( d_1(f) \leq d_1(f|_F) \); hence
\( d_{1}(f) = d_{1}(f|_{F}) \) by Proposition~\ref{prop:subvar}.
Thus, it is enough to prove \( v|_{F} \ne 0 \).

There is a bimeromorphic morphism \( \nu \colon S \to Y \)
from a compact K\"ahler manifold \( S \) by \cite{Va}.
Let \( M \to X \times_{Y} S \) be a proper morphism from a compact
K\"ahler manifold \(M\),
giving rise to a bimeromorphic morphism to the main component of
\( X \times_{Y} S \), i.e., the unique component
dominating \( S \).
Let \( \mu \colon M \to X \) be the induced bimeromorphic morphism and
let \( \varpi \colon M \to S \) be the induced fiber space.
We may replace \( F \) with a general fiber of \( \pi \), since
\( d_{1}(f|_{F}) \) depends only on the cohomology class
\( \class(F) \in \OH^{d, d}(X, \BRR) \) (cf.\ \eqref{eq:prop:subvar}).
Hence we may assume that \( \nu \) is an isomorphism over a neighborhood of
\( \pi(F) \) and that \( \mu \) is an isomorphism along \( \mu^{-1}(F) \).
Let \( f_{M} = \mu^{-1} \circ f \circ \mu
\colon M \ratmap M \) be the meromorphic endomorphism and
let \( \varphi \colon Z \to M \) be a bimeromorphic morphism from
another compact K\"ahler manifold \( Z \) such that
\( g := f_{M} \circ \varphi \colon Z \to M\) is holomorphic.
Let \( (f_{M}^{*})^{(i, i)} \) be an endomorphism of \(\OH^{i, i}(M, \BRR)
\) for \( 0 \leq i \leq n = \dim X\) defined as
\[ (f_{M}^{*})^{(i, i)} \colon \OH^{i, i}(M, \BRR) \xrightarrow{g^{*}}
\OH^{i, i}(Z, \BRR) \xrightarrow{\varphi_{*}} \OH^{i, i}(M, \BRR). \]
Since \( \varphi_{*} \circ \varphi^{*} = \id\),
we have a commutative diagram
\[ \begin{CD}
\OH^{i, i}(X, \BRR) @>{f^{*}}>> \OH^{i, i}(X, \BRR) \\
@V{\mu^{*}}VV @V{\mu^{*}}VV \\
\OH^{i, i}(M, \BRR) @>{(f_{M}^{*})^{(i, i)}}>> \phantom{.}\OH^{i, i}(M, \BRR).
\end{CD} \]
Thus, \( \mu^{*}v \in \overline{P^{1}(M)}\) is also an eigenvector of
\( (f_{M}^{*})^{(1, 1)} \) with the eigenvalue \( d_{1}(f) \).
We have also \( f_{M}^{*} \circ \varpi^{*} = \varpi^{*}\) from
\( \varpi \circ f_{M} = \varpi \).
For any \( \alpha \in \OH^{1, 1}(X, \BRR) \),
\( \beta \in \OH^{1, 1}(S, \BRR) \), and
\( \xi \in \OH^{i, i}(M, \BRR)\), we have the equalities
\begin{align}
(f_{M}^{*})^{(i+1, i+1)}(\xi \cup \mu^{*}\alpha) &=
(f_{M}^{*})^{(i, i)}(\xi) \cup \mu^{*}f^{*}\alpha, \label{eq:fMprojformula}\\
(f_{M}^{*})^{(i+1, i+1)}(\xi \cup \varpi^{*}\beta) &=
(f_{M}^{*})^{(i, i)}\xi \cup \varpi^{*}\beta \notag
\end{align}
by applying the projection formula to
\[ \varphi_{*}(g^{*}(\xi \cup \mu^{*}\alpha)) =
\varphi_{*}(g^{*}\xi \cup \varphi^{*}\mu^{*}f^{*}\alpha),
\quad \text{ and } \quad
\varphi_{*}(g^{*}(\xi \cup \varpi^{*}\beta)) =
\varphi_{*}(g^{*}\xi \cup \varphi^{*}\varpi^{*}\beta). \]
Assuming \( v|_{F} = 0 \), we shall derive a contradiction
by an argument similar to \cite[Section 2.1, Remark (11)]{Zh}.
We set \( x := \mu^{*}\alpha \) and
\( y := \varpi^{*}\beta \in \OH^{1, 1}(M, \BRR) \)
for \( \alpha \in P^{1}(X) \) and \( \beta \in P^{1}(S) \). Then
\( y^{d} = c \class(\mu^{-1}F) \in \OH^{d, d}(M, \BRR) \)
for some \( c > 0 \).
Thus, \( \mu^{*}v \cup y^{d} = 0 \).
Let \( 1 \leq s \leq d\) be the minimum integer such that
\( \mu^{*}v \cup y^{s} = 0 \).
By \cite[Corollaire 3.5]{DS}, there is
a constant \( b \) such that
\[ (b \mu^{*} v + y) \cup y^{s - 1} \cup
\mu^{*}(\alpha_{1} \cup \cdots \cup \alpha_{n - s - 1})
= 0 \in \OH^{n - 1, n - 1}(M, \BRR)\]
for any \( \alpha_{i} \in \OH^{1, 1}(X, \BRR) \),
where \( n = \dim X \).
Taking \( f_{M}^{*} \), by \eqref{eq:fMprojformula}, we have
\[ (b d_{1}(f) \mu^{*}v + y) \cup y^{s - 1} \cup
\mu^{*}(f^{*}\alpha_{1} \cup \cdots \cup f^{*}\alpha_{n - s - 1}) = 0. \]
In particular,
\[ (b \mu^{*} v + y) \cup y^{s - 1} \cup x^{n - s} =
(b d_{1}(f) \mu^{*}v + y) \cup y^{s - 1} \cup x^{n - s} = 0. \]
If \( b = 0 \), then \( y^{s} \cup x^{n - s} = 0 \)
contradicting Lemmas~\ref{lem:positive} and \ref{lem:genfin},
since \( y^{s} \ne 0 \). Therefore,
\( b \ne 0 \) and \( \mu^{*}v \cup y^{s - 1} \cup x^{n - s} =  0 \).
Then \( \mu^{*}v \cup y^{s-1} = 0 \) by
Lemmas~\ref{lem:positive} and \ref{lem:genfin}.
This contradicts the minimality of \( s \).
Thus, we have completed the proof of \eqref{ThD:1} of Theorem~D.
\end{proof}

\subsection{Topological entropies}
\label{subsect:topent}

Let \( f \colon M \to M \) be a surjective endomorphism of a compact
K\"ahler manifold \( M \).
We consider the properties of topological entropy \( h_{\topo}(f) \)
of \( f \).
Instead of giving the definition of \( h_{\topo} \), we use the
following:

\begin{fact}[Gromov \cite{G}, Yomdin \cite{Y} (cf.\ \cite{Fr})]
\label{fact:topent}
For the topological entropy
\( h_{\topo}(f) \) of \( f \), one has
\[ h_{\topo}(f) = \max\nolimits_{1 \leq i \leq n} \log d_{i}(f)
= \log \rho(\OH^{*}(M, \BCC), f^{*}). \]
\end{fact}

As a consequence of Proposition~\ref{prop:subvar}, we have:

\begin{cor}\label{cor:topentsubvar}
In the situation of \emph{Proposition~\ref{prop:subvar}},
\[ h_{\topo}(f) \geq \log \left( \deg(f)\deg(h)^{-1}\right) +
h_{\topo}(h) \geq h_{\topo}(h). \]
\end{cor}

\begin{nota}\label{nota:eigenvalue} Let \( \varphi \colon V \to V \)
be an endomorphism of a finite-dimensional \( \BCC \)-vector
space \( V \). We denote by \( \Lambda(V, \varphi) \)
the set of eigenvalues of \( \varphi \).
For \( \lambda \in \BCC \), we set
\[
V_{\lambda} = V_{\lambda, \varphi} := \bigcup\nolimits_{l \geq 1}
\Ker(\lambda\id_{V} - \varphi)^{l}.\]
\end{nota}

\begin{remn}
If \( \lambda \in \Lambda(V, \varphi) \), then
\( V_{\lambda} \) is a generalized eigenspace.
We have the decomposition
\( V = \bigoplus\nolimits_{\lambda \in \Lambda(V, \varphi)}
V_{\lambda, \varphi}\),
which is functorial in the following sense:
Let \( h \colon V_{1} \to V_{2} \) be a \( \BCC \)-linear
map of finite-dimensional \( \BCC \)-vector spaces.
Let \( \varphi_{i} \colon V_{i} \to V_{i} \) be an endomorphism for \(
i = 1 \), \( 2 \) such that \( h \circ \varphi_{1} = \varphi_{2}
\circ h \). Then \( h = \bigoplus h_{\lambda} \) for
\( h_{\lambda} \colon (V_{1})_{\lambda, \varphi_{1}} \to
(V_{2})_{\lambda, \varphi_{2}}\).
\end{remn}

\begin{lem}\label{lem:irrdec}
Let \( Y \) be a reduced compact complex analytic space and
\( \mu \colon X \to Y \) a proper surjective morphism from another
reduced compact complex analytic space \( X \) such that the
restriction \( \mu^{-1}(U) \to U \) of \( \mu \) is a smooth K\"ahler morphism
for a dense Zariski-open subset \( U \subset Y \).
Let \( g \colon Y \to Y \) and \( f \colon X \to X \) be surjective
endomorphisms such that \( \mu \circ f = g \circ \mu \) and
\( g^{-1}(U) = U \).
Let \( g_{B} \colon B \to B\) and
\( f_{A} \colon A\to A \) be the induced endomorphisms
for the complement \( B = Y \setminus U \)
and \( A = \mu^{-1}(B) \), respectively.
Then one has the following inclusion for any
\( p \geq 0\)\emph{:}
\[ \Lambda(\OH^{p}(Y, \BCC), g^{*}) \subset
\Lambda(\OH^{p}(X, \BCC), f^{*}) \cup
\Lambda(\OH^{p}(B, \BCC), g_{B}^{*}) \cup
\Lambda(\OH^{p - 1}(A, \BCC), f_{A}^{*}). \]
\end{lem}

\begin{proof}
We have a commutative diagram
\[ \begin{CD}
\cdots @>>> \OH^{i}_{c}(U, \BZZ) @>>> \OH^{i}(Y, \BZZ) @>>>
\OH^{i}(B, \BZZ) @>>> \cdots \\
@. @V{\mu^{*}}VV @V{\mu^{*}}VV
@V{\mu|_{A}^{*}}VV \\
\cdots @>>> \OH^{i}_{c}(\mu^{-1}(U), \BZZ) @>>> \OH^{i}(X, \BZZ) @>>>
\OH^{i}(A, \BZZ) @>>> \cdots
\end{CD} \]
of exact sequences, which are compatible with the actions of \( g^{*} \)
and \( f^{*} \).
Since
\( \mu^{-1}(U) \to U \) is a smooth K\"ahler morphism,
we have a quasi-isomorphism
\[ \BRR \mu_{*}\BCC_{X}|_{U} \sim_{\qis} \bigoplus
\OR^{i}\mu_{*}\BCC_{X}|_{U}[-i]\]
in the derived category on \( U \),
by the hard Lefschetz theorem on fibers and by \cite{D}.
In particular, the homomorphism
\( \mu^{*} \colon \OH^{i}_{c}(U, \BCC) \to
\OH^{i}_{c}(\mu^{-1}(U), \BCC) \)
is injective for any \( i \).
For a complex number \( \lambda \),
if \( \OH^{p}(Y, \BCC)_{\lambda, g^{*}} \ne 0 \) and
if \( \OH^{p}(B, \BCC)_{\lambda, g_{B}^{*}} =
\OH^{p}(X, \BCC)_{\lambda, f^{*}} = 0 \), then
\( \OH^{p-1}(A, \BCC)_{\lambda, f_{A}^{*}} \ne 0 \)
by the commutative diagram above.
Thus, the assertion follows.
\end{proof}

\begin{cor}\label{cor:irrdec}
Let \( Y \) be a reduced compact complex analytic space with a
surjective endomorphism \( g \colon Y \to Y \).
Then there exist a finite set \( \{Z_{i}\}_{1 = 1}^{N} \)
of closed subvarieties and a positive integer \( k \) such that
\( Y = \bigcup Z_{i} \), \( (g^{k})^{-1}(Z_{i}) = Z_{i} \) for any \(
1 \leq i \leq N\), and that the following inclusion holds for
any \( p \geq 0 \)\emph{:}
\[ \Lambda(\OH^{p}(Y, \BCC), (g^{k})^{*}) \subset
\bigcup\nolimits_{i = 1}^{N}\bigcup\nolimits_{q = 0}^{p}
\Lambda(\OH^{q}(Z_{i}, \BCC), (g^{k}|_{Z_{i}})^{*}) . \]
\end{cor}

\begin{proof}
Let \( \{Y_{i}\} \)
be the set of the irreducible components of \( Y
\) and let \( \tau \colon X := \bigsqcup Y_{i} \to Y \) be the natural
morphism. Let \( B \subset Y \) be the minimum closed subset such
that \(  X \setminus \tau^{-1}B \to Y \setminus B\) is an
isomorphism. By replacing \( g \) with a power \( g^{k} \), we may
assume that \( g^{-1}Y_{i} = Y_{i} \) for any \( i \). Then,
an endomorphism \(f \colon X \to X\) with
\(\tau \circ f = g \circ \tau\) is induced,
where \(g^{-1}(B) = B\) by the minimality of \(B\) and
\(f^{-1}(\tau^{-1}(B)) = \tau^{-1}(B)\).
Applying Lemma~\ref{lem:irrdec} to
\(X \to Y\), \(f\), \(g\), \(B\), and \(\tau^{-1}(B)\),
we have the following inclusion
for any \( p \):
\[ \Lambda(\OH^{p}(Y, \BCC), g^{*}) \subset
\Lambda(\OH^{p}(B, \BCC), g^{*}) \cup
\bigcup\nolimits
\left( \Lambda(\OH^{p}(Y_{i}, \BCC), g^{*}) \cup
\Lambda(\OH^{p-1}(Y_{i} \cap B, \BCC), g^{*}) \right). \]
Applying the same argument to \( B \) and
\( Y_{i} \cap B \), and continuing, we have the assertion.
\end{proof}

\begin{prop}\label{prop:entreducible}
Let \( Y \) be a reduced compact complex analytic space and let
\( \phi \colon M \to Y \) be a proper surjective morphism from a
finite disjoint union \( M \) of compact K\"ahler manifolds \(
M_{\alpha} \) such that \( \phi(M_{\alpha}) \) is an irreducible
component of \( Y \).
Let \( g \colon Y \to Y \) and \( f \colon M \to M \)
be \'etale surjective endomorphisms such that
\( \phi \circ f = g \circ \phi \).
If \( f^{-1}(M_{\alpha}) = M_{\alpha} \)
for any \( \alpha \), then, for the induced endomorphisms \( f|_{M_{\alpha}}
\colon M_{\alpha} \to M_{\alpha} \), one has
\[ \rho(\OH^{*}(Y, \BCC), g^{*}) \leq
\max\nolimits_{\alpha}
\rho(\OH^{*}(M_{\alpha}, \BCC), (f|_{M_{\alpha}})^{*}) =
\max\nolimits_{\alpha, l}
d_{l}(f|_{M_{\alpha}}) . \]
\end{prop}

\begin{proof}
We shall prove by induction on \( \dim M = \max \{\dim M_{\alpha}\} \).
If \( \dim M = 0 \), then the assertion holds by a trivial reason.
Assume that the assertion holds when the dimension is less
than \( \dim M \).
If the inequality of the spectral radius holds for a power \( g^{l}
\), then it holds also for \( g \);
in fact, \( \rho(\OH^{p}(Y, \BCC), (g^{l})^{*})
= \rho(\OH^{p}(Y, \BCC), (g^{})^{*})^{l}\).
Thus, we may replace \( g \) with any power \( g^{l} \).
Let \(U \subset Y\) be the maximal Zariski open subset
such that \( \phi^{-1}(U) \to U \) is smooth.
For the complement \( B = Y \setminus U \) and \( A = \phi^{-1}(B) \),
we have \( g^{-1}(B) = B \) and \( f^{-1}(A) = A \)
by the maximality of \( U \) and by the \'etaleness of
\( g \) and \( f \).
Thus, we can apply Lemma~\ref{lem:irrdec}, and obtain
\[ \rho(\OH^{*}(Y, \BCC), g^{*}) \leq
\max \{ \rho(\OH^{*}(M, \BCC), f^{*}),
\rho(\OH^{*}(B, \BCC), (f|_{B})^{*}),
\rho(\OH^{*}(A, \BCC), (f|_{A})^{*}) \}.\]
Let \( A = \bigcup A_{\beta} \) be the irreducible decomposition.
Replacing \( g \) with its power, we may assume that
\( f^{-1}(A_{\beta}) = A_{\beta} \) for any \( \beta \).
Let \( Z_{\beta} \to A_{\beta} \) be an equivariant resolution of
singularities of \( A_{\beta} \) with respect to the \'etale
endomorphism \( f|_{A_{\beta}} \). Let \( Z \) be the disjoint union
\( \bigsqcup Z_{\beta} \) and let \( h \colon Z \to Z \) be the induced
\'etale endomorphism, i.e., \( \psi \circ h = f|_{A} \circ \psi \)
for the induced morphism \( \psi \colon Z \to A \).
Then \( (Z \to A \to B, h, g|_{B}) \) and \( (Z \to A, h, f|_{A}) \)
satisfy the same condition of Proposition~\ref{prop:entreducible}
as \( (M \to Y, f, g) \).
By induction, we have
\[ \max \{ \rho(\OH^{*}(B, \BCC), (g|_{B})^{*}),
\rho(\OH^{*}(A, \BCC), (f|_{A})^{*}) \} \leq
\max\nolimits_{\beta}\rho(\OH^{*}(Z_{\beta}, \BCC), (h|_{Z_{\beta}})^{*})\]
after replacing \( g \) with a power \( g^{l} \). On the other hand,
\[ \rho(\OH^{*}(Z_{\beta}, \BCC), (h|Z_{\beta})^{*})
= \exp(h_{\topo}(h|_{Z_{\beta}})) \leq
\exp (h_{\topo}(f|_{M_{\alpha}})) =
\rho(\OH^{*}(M_{\alpha}, \BCC), (f|M_{\alpha})^{*}) \]
for \( A_{\beta} \subset M_{\alpha}  \),
by Proposition~\ref{prop:subvar}.
Hence,
\[ \rho(\OH^{*}(Y, \BCC), g^{*}) \leq \max\nolimits_{\alpha}
\rho(\OH^{*}(M_{\alpha}, \BCC), (f|M_{\alpha})^{*}). \qedhere\]
\end{proof}

\begin{lem}\label{lem:incl}
Let \( \pi \colon X \to Y \) be a smooth fiber space of
complex manifolds and \( f \colon X \to X \) a surjective
endomorphism such that \( \pi \circ f = \pi  \).
Let \( F \) be a fiber of \( \pi \).
Then, for any \( p \geq 0 \), one has an inclusion
\[  \Lambda(\OH^{p}_{c}(X, \BCC), f^{*}) \subset
\bigcup\nolimits_{q = 0}^{p}
\Lambda(\OH^{q}(F, \BCC), (f|_{F})^{*}). \]
\end{lem}

\begin{proof}
We have an endomorphism \( f^{*} \) of the complex
\( \OR \pi_{*}\BCC_{X} \) in the derived category of sheaves of abelian
groups on \( Y \) such that the induced endomorphism
of \( \OR\Gamma_{c}(Y, \OR \pi_{*}\BCC_{X}) \)  coincides with that of
\( \OR\Gamma_{c}(X, \BCC) \).
In particular, the Leray spectral sequence
\[ E_{2}^{p, q} = \OH^{p}_{c}(Y, \OR^{q}\pi_{*}\BCC_{X}) \Rightarrow
E^{p+q} = \OH^{p+q}_{c}(X, \BCC)\]
admits the endomorphisms \( f^{*} \colon E_{r}^{p, q} \to E_{r}^{p, q}
\) compatible with \( f^{*} \) on \( E^{p+q} \).
For a complex number \( \lambda \),
let \( (\OR^{i}\pi_{*}\BCC_{X})_{\lambda} \) be
the union of the kernels of
\[ (f^{*} - \lambda \id)^{l} \colon
\OR^{i}\pi_{*}\BCC_{X} \to \OR^{i}\pi_{*}\BCC_{X} \]
for all \( l > 0\) (cf.\ Notation~\ref{nota:eigenvalue}).
Then,
\[ (E_{2}^{p, q})_{\lambda} \isom \OH^{p}_{c}(Y,
(\OR^{q}\pi_{*}\BCC_{X})_{\lambda}) \]
for any \( p \), \( q \), and \( \lambda \).
Since \( E_{2}^{p, q} \Rightarrow E^{p+q} \) is decomposed
into \((E_{2}^{p, q})_{\lambda}\Rightarrow (E^{p+q})_{\lambda}  \), we have
\[ \Lambda(\OH^{p}_{c}(X, \BCC), f^{*}) \subset
\bigcup\nolimits_{q = 0}^{p} \{\lambda \in \BCC \mid
(\OR^{q}\pi_{*}\BCC_{X})_{\lambda} \ne 0 \} =
\bigcup\nolimits_{q = 0}^{p}\Lambda(\OH^{q}(F, \BCC), (f|_{F})^{*}). \qedhere\]
\end{proof}

We are ready to prove the second assertion of Theorem~D.

\begin{proof}[Proof of \eqref{ThD:2} of \emph{Theorem~D}]
We may replace \( f \) with a power \( f^{k} \) since \(
h_{\topo}(f^{k}) = kh_{\topo}(f) \) and
\( h_{\topo}(f^{k}|_{F}) = k h_{\topo}(f|_{F}) \).

\emph{Step}~1. The first reduction:
Let \( \nu \colon Y' \to Y \) be a bimeromorphic morphism from a
compact K\"ahler manifold \( Y' \). Then \( f \) induces an \'etale
endomorphism \( f \times_{Y} Y' \colon X \times_{Y} Y' \to X
\times_{Y} Y' \).
Let \( X' \) be the main component of \( X \times_{Y} Y' \).
Thus, there exists an equivariant resolution of singularities
\( \widetilde{X} \to X' \) with respect to the \'etale endomorphism.
We may assume \(\widetilde{X}\) to be K\"ahler since
\(\widetilde{X} \to X'\) can be chosen as a projective morphism.
Let \( \tilde{f} \) be the induced \'etale endomorphism of
\( \widetilde{X} \). Then \( d_{l}(\tilde{f}) = d_{l}(f) \) by
Lemma~\ref{lem:d=genfin}.
Thus, we may assume that \( Y \) is also a compact K\"ahler manifold
and \( \pi \colon X \to Y \) is smooth over
the complement \( Y \setminus D \) for a divisor \( D \subset Y \).

\emph{Step}~2: A setting of the proof.
Let \( \sigma \colon S \to Y \) be a flattening of \( \pi \), i.e., the main
component of \( X \times_{Y} S \) is flat over \( S \) (cf.\ \cite{H}).
Here, we may assume that \( \sigma^{-1}D \) is also
a normal crossing divisor.
Let \( M \to X \times_{Y} S \) be a bimeromorphic morphism from a
compact K\"ahler manifold \( M \) which is given as an
equivariant resolution of singularities of the main component with
respect to the induced \'etale endomorphism
\( f \times_{Y} \id_{S} \colon X \times_{Y} S \to X \times_{Y} S \).
Let \( \mu \colon M \to X \) be the induced bimeromorphic morphism,
\( \varpi \colon M \to S \) the induced fiber space, and
\( g \colon M \to M \) the induced \'etale endomorphism, i.e.,
\( \mu \circ g = f \circ \mu \) and \( \varpi \circ g = \varpi \).
We set \( E = \pi^{*}(D)_{\red} \). For the prime
decomposition \( E = \sum\nolimits_{j = 1}^{l} E_{j} \),
there is a positive integer \( k
\) such that \( (f^{k})^{-1}E_{j} = E_{j} \).
Here, we may assume \( k = 1 \) by replacing \( f \) with \( f^{k} \);
hence \( f^{-1}(E_{j}) = E_{j} \).
Let \( G_{j} \) be the proper transform of \( E_{j} \) in \( M \)
for any \( j \). Then \( g^{-1}(G_{j}) = G_{j} \).
By a suitable choice of an equivariant embedded resolution of
singularities, we may assume that all \( G_{j} \) are nonsingular.
Then \( \varpi(G_{j}) \) is a prime component of
\( \sigma^{-1}D \) for any \( j \) by the flattening.
Let \( h_{j} \colon G_{j} \to G_{j} \) be the \'etale endomorphism
\( g|_{G_{j}} \).
In order to prove \eqref{ThD:2} of Theorem~D,
it is enough to show the following two inequalities:
\begin{gather}
h_{\topo}(f) \leq \max
\{h_{\topo}(f|_{F}), h_{\topo}(h_{j})\, (1 \leq j \leq l)\}.
\label{eq:final1}\\
h_{\topo}(f|_{F}) \geq h_{\topo}(h_{j}) \qquad \text{ for any } \quad
1 \leq j \leq l. \label{eq:final2}
\end{gather}
Indeed, we have already the other inequality
\( h_{\topo}(f) \geq h_{\topo}(f|_{F}) \) in Corollary~\ref{cor:topentsubvar}.
The inequality \eqref{eq:final1} is shown in \emph{Step}~3 below, while
the other inequality \eqref{eq:final2} is shown in \emph{Step}~4
by using induction on the dimension of \( X \).

\emph{Step}~3. Proof of \eqref{eq:final1}:
The natural long exact sequence
\[ \cdots \to \OH^{p}_{c}(X \setminus E, \BCC) \to
\OH^{p}(X, \BCC) \to \OH^{p}(E, \BCC) \to \cdots \]
admits an endomorphism \( f^{*} \) induced from \( f \).
Thus,
\begin{align*}
\Lambda(\OH^{*}(X, \BCC), f^{*}) &\subset
\Lambda(\OH^{*}_{c}(X \setminus E, \BCC), (f|_{X \setminus E})^{*}) \cup
\Lambda(\OH^{*}(E, \BCC), (f|_{E})^{*}) \\
&\subset
\Lambda(\OH^{*}(F, \BCC), (f|_{F})^{*}) \cup
\bigcup\nolimits_{j}
\Lambda(\OH^{*}(G_{j}, \BCC), h_{j}^{*})
\end{align*}
by Proposition~\ref{prop:entreducible} applied to \( \bigsqcup G_{j} \to E \)
and Lemma~\ref{lem:incl} applied to \( X \setminus E \to Y \setminus D \).
Therefore, \eqref{eq:final1} follows.

\emph{Step}~4. Proof of \eqref{eq:final2}: By induction, we may
assume that \eqref{ThD:2} of Theorem~D holds
for lower dimensional ambiant spaces.
Let \( P  \) be a general point of \( \varpi(G_{j}) \). We may assume
that \( P \) is a nonsingular point of \( \sigma^{-1}D \), \(
\varpi \) is flat over an open neighborhood of \( P \), and \(
G_{j} \to \varpi(G_{j}) \) is smooth over an open neighborhood of \( P
\). Let \( G_{j} \to S_{j} \) be the fiber space obtained as the Stein
factorization of \( G_{j} \to \varpi(G_{j}) \).
Let \( \Gamma \) be the fiber of \( G_{j} \to S_{j} \) over a point of
\( S_{j} \) lying over \( P \). By induction, we may apply \eqref{ThD:2}
of Theorem~D
to the fiber space \( G_{j} \to S_{j} \) and the \'etale endomorphism
\( g|_{G_{j}} \). Thus,
\( h_{\topo}(h_{j}) = h_{\topo}(g|_{\Gamma}) \).
Let \( \class(\Gamma) \) and \( \class(\varpi^{-1}(P)) \) be the cohomology
classes in \( \OH^{d, d}(M, \BRR) \) corresponding to the subspaces
\( \Gamma \) and \( \varpi^{-1}(P) \), respectively, where \( d = \dim S = \dim Y \).
Then \( \class(\varpi^{-1}(P)) - \class(\Gamma)
\) is represented by a positive current.
The class \( \class(\varpi^{-1}(P)) \in \OH^{d, d}(M, \BRR) \) equals the
class \( \class(F') \) corresponding to a smooth fiber \( F' \)
of \( \varpi \colon M \to S \), since \( \varpi \) is flat
along \( \varpi^{-1}(P) \). Thus,
\begin{align*}
\delta_{l}((g|_{\Gamma})^{m}, \eta|_{\Gamma}) =
\delta_{l}(g^{m}, \eta; \class(\Gamma)) &\leq
\delta_{l}(g^{m}, \eta; \class(\varpi^{-1}(P))) \\
&= \delta_{l}(g^{m}, \eta; \class(F')) =
\delta_{l}((g|_{F'})^{m}, \eta|_{F'})
\end{align*}
for a K\"ahler form \( \eta \), any \( 0 \leq l \leq n - d  \), and
any \( m \geq 1 \).
In particular, \( h_{\topo}(g|_{\Gamma}) \leq h_{\topo}(g|_{F'}) \).
We may assume that \( \mu \colon M \to X \) is an isomorphism along \(
F' \) and \( \sigma \colon S \to Y \) is an isomorphism at the point
\( \varpi(F') \). Therefore, \( h_{\topo}(h_{j}) = h_{\topo}(g|_{\Gamma})
\leq h_{\topo}(g|_{F'}) = h_{\topo}(f|_{F}) \).
Thus, we have completed the proof of \eqref{ThD:2} of Theorem~D.
\end{proof}

\end{document}